\documentclass[11pt,twoside]{amsart}
\usepackage{bbm}
\usepackage{enumerate}
\usepackage{graphicx,color}
\usepackage[all]{xy}
\usepackage{relsize}
\usepackage{verbatim}
\usepackage[dvipsnames]{xcolor}
\usepackage{amsmath}

\usepackage{eso-pic}

\theoremstyle{plain}
\newtheorem{The}{Theorem}
\newtheorem*{The*}{Theorem}
\newtheorem{Pro}{Proposition}
\newtheorem{Lem}{Lemma}
\newtheorem{Cor}{Corollary}
\newtheorem*{Cor*}{Corollary}

\theoremstyle{definition}
\newtheorem{Def}{Definition}
\newtheorem{Rem}{Remark}

\newtheorem*{Rem*}{Remark}

\numberwithin{equation}{section}
\numberwithin{The}{section}
\numberwithin{Def}{section}
\numberwithin{Rem}{section}
\numberwithin{Exa}{section}
\numberwithin{Lem}{section}
\numberwithin{Pro}{section}
\numberwithin{Cor}{section}

\DeclareMathOperator{\Tr}{tr}             

\DeclareMathOperator{\Min}{Min}

\DeclareMathOperator{\del}{\partial}

\newcommand{\bigletter}[1]{\mathlarger{#1}}

\newcommand{\bigPi}{\bigletter{\Pi}}
\newcommand{\R}{\mathbb{R}}

\newcommand{\C}{\mathbb{C}}
\newcommand{\N}{\mathbb{N}}
\newcommand{\Z}{\mathbb{Z}}

\renewcommand{\H}{\mathbb{H}}

\newcommand{\bigslant}[2]{{\raisebox{.2em}{$#1$}\left/\raisebox{-.2em}{$#2$}\right.}}

\begin{document}

\title[Constrained Willmore minimizers of non-rectangular conformal class]{First explicit constrained Willmore minimizers of non-rectangular conformal class}

\author{Lynn Heller}

\address{ Institut f\"ur Differentialgeometrie\\  Leibniz Universit{\"a}t Hannover\\ Welfengarten
1\\ 30167 Hannover\\ Germany}
 
 \email{lynn.heller@math.uni-hannover.de}
 \author{Cheikh Birahim Ndiaye}

\address{Department of Mathematics of Howard University\\
204 Academic Support Building B
Washington, DC 20059k \\ USA
 }
 \email{cheikh.ndiaye@howard.edu}

 \date{\today}

\thanks{The first  author wants to thank the Ministry of Science, Research and Art Baden-W\"uttemberg and the European social Fund for supporting her research within the Margerete von Wrangell Programm. Further, both authors are indebted to the Baden-W\"urttemberg foundation for supporting the project within the Eliteprogramm for Postdocs. Moreover, the second author thank the Deutsche Forschungsgemeinschaft (DFG) for financial support through the project ''Fourth-order uniformization type theorem for $4$-dimensional Riemannian manifolds'' and the Schweizerischer Nationalfonds (SNF) for financial support through the project PP00P2\_144669}
\noindent

\maketitle

\begin{abstract} 
We study immersed tori in $3$-space minimizing the Willmore energy in their respective conformal class. Within the rectangular conformal classes $\;(0,b)\;$ with $\;b \sim 1\;$ the homogenous tori $\;f^b\;$ are known to be the unique constrained Willmore minimizers (up to invariance). In this paper 
we generalize this result and show that the candidates constructed in \cite{HelNdi2} are indeed constrained Willmore minimizers in certain non-rectangular conformal classes $\;(a,b).\;$ 
 Difficulties arise from the fact that these minimizers are non-degenerate for  $\;a \neq 0\;$ but smoothly converge to the degenerate homogenous tori $\;f^b\;$ as $\;a \longrightarrow 0.\;$
 As a byproduct of our arguments, we show  that the minimal Willmore energy $\;\omega(a,b)\;$ is real analytic and concave in $\;a \in (0, a^b)\;$ for some $\;a^b>0\;$ and fixed $\;b \sim 1,\;$ $b \neq 1.$

 \end{abstract}

\setcounter{tocdepth}{1}

 \tableofcontents


\section{Introduction and statement of the results}
In the 1960s Willmore \cite{Willmore}  proposed to study the critical values and critical points of the {\em bending energy} 
\[
\mathcal{W}(f)=\int_M H^2dA,
\]

the average value of the squared mean curvature $\;H\;$ of an immersion $\;f\colon M\longrightarrow \R^3\;$ of a closed surface  $\;M.\;$ In this definition we denote by $\;dA\;$ the induced volume form and  $\;H:= \tfrac{1}{2}tr (\text{II})\;$ with II the second fundamental form of the immersion $\;f.\;$
Willmore showed that the absolute minimum of this functional is attained at round spheres with Willmore energy $\;\mathcal{W}=4\pi.\;$ He also conjectured that the minimum over surfaces of genus $1$ is attained at (a suitable stereographic projection of) the Clifford torus in the 3-sphere with $\;\mathcal{W}=2\pi^2.\;$
It soon was noticed that the bending energy $\;\mathcal{W}\;$ (by then also known as the {\em Willmore energy}) is invariant under M\"obius transformations of the target space -- in fact, it is invariant under conformal changes of the metric in the target space, see \cite{Blaschke, Chen}. Thus, it makes no difference for the study of the Willmore functional which constant curvature  target space is chosen. \\

Bryant \cite{Bryant} characterized all Willmore spheres as M\"obius transformations of genus $\;0\;$ minimal surfaces in $\;\R^3\;$ with planar ends. The value of the bending energy on Willmore spheres is thus quantized to be $\; \mathcal{W}=4\pi k,\;$ with $\;k\geq 1\;$ the number of ends. With the exception of  $\;k=2,3,5,7\;$ all values occur. For more general target spaces the variational setup to study this surfaces can be found in \cite{Mondino}.
The first examples of Willmore surfaces not M\"obius equivalent to minimal surfaces were found by Pinkall \cite{Pinkall}. They were constructed via lifting elastic curves $\;\gamma\;$ with geodesic curvature $\;\kappa\;$ on the 2-sphere under the Hopf fibration to Willmore tori in the 3-sphere, where
elastic curves are the critical points for the elastic energy 

$$E(\gamma)= \int_{\gamma} (\kappa^2+1)ds$$

 and $\;s\;$ is the arclength parameter of the curve. Later Ferus and Pedit \cite{FerusPedit} classified all Willmore tori equivariant under a M\"obius $\;S^1$-action on the 3-sphere (for the definition of $\;S^1$-action see \cite{He1}).\\

The Euler-Lagrange equation for the Willmore functional 

\[
{\Delta} H+ 2H(H^2-K)=0,
\]

 where $\;K\;$ denotes the Gau\ss ian curvature of the surface $\;f\colon M\longrightarrow \R^3\;$ and $\;\Delta\;$ its Laplace-Beltrami operator, is a 4th order elliptic PDE for $\;f\;$ since the mean curvature vector $\;\vec{H}\;$ is the normal part of $\;\Delta f.\;$ Its analytic properties are prototypical for non-linear bi-Laplace equations. Existence of a minimizer for the Willmore functional $\;\mathcal{W}\;$ on the space of smooth immersions from 2-tori was shown by Simon \cite{Simon}. Bauer and Kuwert \cite{BauerKuwert} generalized this result to higher genus surfaces. After a number of partial results, e.g. \cite{LiYau}, \cite{MontielRos}, \cite{Ros}, \cite{Top}, \cite{FLPP},  Marques and Neves  \cite{MarquesNeves}, using Almgren-Pitts min-max theory, gave a proof of the Willmore conjecture in $3$-space in 2012. An alternate strategy was proposed in \cite{Schmidt}\\

A more refined, and also richer, picture emerges when restricting the Willmore functional to the subspace of smooth immersions $\;f\colon M\longrightarrow \R^3\;$ inducing a given conformal structure on $\;M.\;$  Thus, $\;M\;$  now is  a Riemann surface and we study the Willmore energy $\;\mathcal{W}\;$ on the space of smooth conformal immersions $\;f\colon M\longrightarrow \R^3\;$ whose critical points are called {\em (conformally) constrained Willmore surfaces}. The conformal constraint augments the Euler-Lagrange equation by  $\;\omega\in H^0(K^2_M)\;$ paired with the trace-free second fundamental  form $\;\mathring{\text{II}}\;$ of the immersion

\begin{equation}\label{CWEL}
\Delta H+ 2H(H^2-K)=\,<\omega,\mathring{\text{II}}>
\end{equation}

with $\;H^0(K^2_M)\;$ denoting the space of holomorphic quadratic differentials. In the Geometric Analytic literature, the space $\;H^0(K_M^2)\;$ is also referred to as $\;S_2^{TT}\big(g_{euc}\big)\;$ the space of symmetric, covariant, transverse and traceless $2$-tensors with respect to the euclidean metric $\;g_{euc}.\;$
Since there are no holomorphic (quadratic) differentials on a genus zero 
Riemann surface, constrained Willmore spheres are the same as Willmore spheres. For higher genus surfaces this is no longer the case: constant mean curvature surfaces (and their M\"obius transforms) are constrained Willmore, as one can see by choosing $\;\omega:=\mathring{\text{II}}\;$ as the holomorphic Hopf differential in the Euler Lagrange equation \eqref{CWEL}, but not Willmore unless they are minimal in a space form. 
Bohle \cite{Bohle}, using techniques developed in \cite{BohLesPedPin} and \cite{BohPedPin_ana}, showed that all constrained Willmore tori have finite genus spectral curves and are described by linear flows on the Jacobians of those spectral curves\footnote{For the notion of spectral curves and the induced linear flows on the Jacobians see \cite{BohLesPedPin}.}. Thus the complexity of the map $\;f\;$ heavily depends on the genus  its spectral curve $\;\Sigma$ -- the spectral genus -- giving the dimension of the Jacobian of $\;\Sigma\;$ and thus codimension of the linear flow. The simplest examples of constrained Willmore tori, which have spectral genus zero, are the tori of revolution in $\;\R^3\;$ with circular profiles -- the homogenous tori. Those are stereographic images of products of circles of varying radii ratios in the 3-sphere and thus have constant mean curvature as surfaces in the 3-sphere.  Starting at the Clifford torus, which has mean curvature $\;H=0\;$ and a square conformal structure, these homogenous tori in the 3-sphere parametrized by their mean curvature $\;H\;$ ``converge'' to a circle as $\;H\longrightarrow \infty\;$ and thereby sweeping out  all rectangular conformal structures. Less trivial examples of constrained Willmore tori come from the Delaunay tori of various lobe counts (the $\;n$-lobed Delaunay tori) in the 3-sphere whose spectral curves have genus $1$, see Figure~\ref{fig:torus-tree} and \cite{KilianSchmidtSchmitt1} for their definition. \\
\begin{figure}
\includegraphics[width=0.5\textwidth]{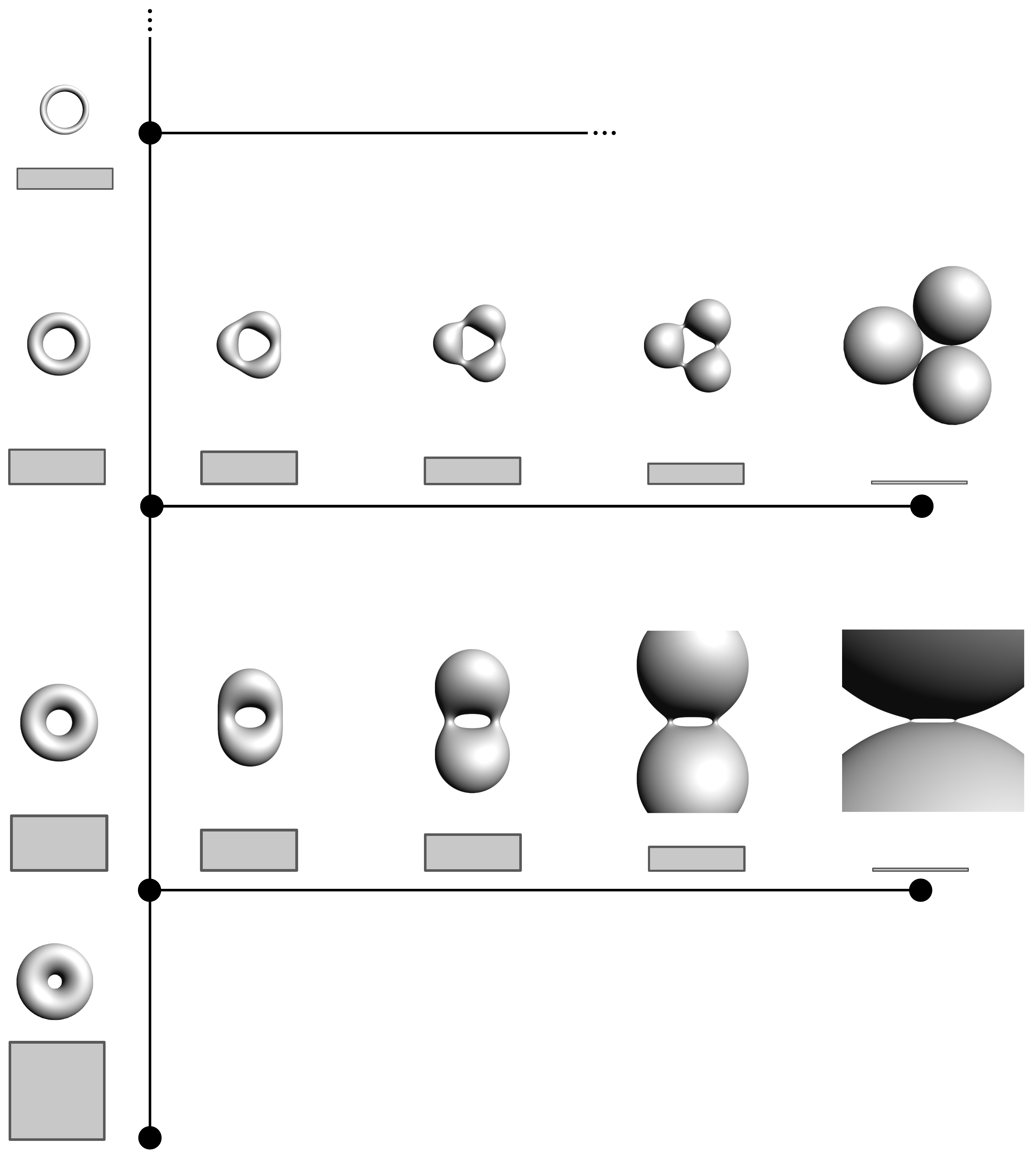}
\caption{
The vertical stalk represents the family of homogenous tori,
starting with the Clifford torus at the bottom.
Along this stalk are bifurcation points from which embedded Delaunay tori continue the homogenous family. The rectangles indicate the conformal types. The family of surfaces starting at the Clifford torus, bifurcating at the first branch point has Willmore energy below $\;8 \pi\;$ and is conjectured to be the minimizer in their respective conformal class.
 Image by Nicholas Schmitt.
}
\label{fig:torus-tree}
\end{figure}

Existence and regularity of a $\;W^{2,2} \cap W^{1, \infty}\;$ minimizer $\;f\colon M\longrightarrow \R^3\;$ for a prescribed Riemann surface structure\footnote{For the notion of $\;W^{2,2} \cap W^{1, \infty}\;$ immersions see \cite{KuwertSchätzle}, \cite{Riviere} or \cite{KuwertLi}.} (constrained Willmore minimizer) was shown by \cite{KuwertSchätzle}, \cite{KuwertLi}, \cite{Riviere2} and \cite{Schätzle} under the assumption that the infimum Willmore energy in the conformal class  is below $\;8\pi.\;$  The latter assumption ensures that minimizers are embedded by the Li and Yau inequality \cite{LiYau}. A broader review of analytic results for Willmore surfaces can be found in the lecture notes \cite{KuwertSchätzle} and \cite{Riviere3}, see also the references therein. \\

Ndiaye and Sch\"atzle \cite{NdiayeSchätzle1, NdiayeSchätzle2} identified the first explicit constrained Willmore minimizers (in every codimension) for rectangular conformal classes in a neighborhood (with size depending on the codimension) of the square class to be the homogenous tori.  These tori of revolution with circular profiles, whose spectral curves have genus 0, eventually have to fail to be minimizing in their conformal class for $\;H >>1,\;$ since their Willmore energy can be made arbitrarily large and any rectangular torus can be conformally embedded into $\;\R^3\;$ (or $\;S^3$) with Willmore energy below $\;8\pi,\;$ see \cite{ KilianSchmidtSchmitt1, NdiayeSchätzle2}. Calculating the 2nd variation of the Willmore energy $\;\mathcal{W}\;$ along homogenous tori Kuwert and Lorenz \cite{KuwertLorenz} showed that zero eigenvalues only appear at those conformal classes whose rectangles have side length ratio $\;\sqrt{k^2-1}\;$ for an integer $\;k\geq 2,\;$ at which the index of the surface increase. These are exactly the rectangular conformal classes from which the $\;k$-lobed Delaunay tori (of spectral genus 1) bifurcate.  Any of the families starting from the Clifford torus, following homogenous tori to the $\;k$-th  bifurcation point, and continuing with the $\;k$-lobed Delaunay tori sweeping out all rectangular classes (see Figure~\ref{fig:torus-tree}) ``converge'' to a neckless of spheres as conformal structure degenerates. The Willmore energy $\;\mathcal W\;$ of the resulting family\footnote{For simplicity we call this family in the following the $\;k$-lobed Delaunay tori.} is strictly monotone and satisfies $\;2\pi^2\leq \mathcal{W}<4\pi k,\;$ see \cite{KilianSchmidtSchmitt1, KilianSchmidtSchmitt2}.  
Thus for $\;k=2\;$ the existence of $2$-lobed Delaunay tori imply that the infimum Willmore energy in every rectangular conformal class is always below $\;8 \pi\;$ and hence there exist embedded constrained Willmore minimizers for these conformal types by \cite{KuwertSchätzle}.
It is conjectured that the minimizers for $\;\mathcal{W}\;$ in rectangular conformal classes are given by the $2$-lobed Delaunay tori. For a more detailed discussion of the $2$-lobe-conjecture see \cite{HePe}.  Surfaces of revolution with prescribed boundary values was studied in \cite{Grunau}. \\

In this paper we turn our attention to finding explicit constrained Willmore minimizer in non-rectangular conformal classes. Putative minimizers were constructed in \cite{HelNdi2}. Our main theorem is the following:\\

\begin{The}[Main Theorem]\label{MainTheorem} $\ $\\
For every $\;b \sim 1\;$ and $ \;b \neq 1\;$ there exists $\;a^b>0\;$ such that for every $\;a \in [0,a^b)\;$  the $(1,2)$-equivariant tori of intrinsic period $1$ (see \cite{HelNdi2} and Figure \ref{1lobe}) with conformal class $\;(a,b)\;$ are constrained Willmore minimizers. Moreover, for $\;b \sim 1\;$ and $\;b \neq 1\;$ fixed, the minimal Willmore energy map 
\begin{equation*}
\begin{split}
\omega(\cdot, b):  \ [&0, a^b) \longrightarrow \R_+,\\ &a \longmapsto \omega(a,b)
\end{split}
\end{equation*}
is concave and for $\;a\neq 0\;$ it is real analytic.\\
\end{The}

 \begin{Def}
Let $\;\bigPi = (\bigPi^1, \bigPi^2)\;$  denote the projection map from the space of immersions to the Teichm\"uller space. For $\;\alpha, \beta \in \R\;$ we use the abbreviations

\begin{equation}
\begin{split}
\mathcal W_{\alpha, \beta}(f) &:= \mathcal W(f) - \alpha \bigPi^1(f) - \beta \bigPi^2(f)\\
\mathcal W_{\alpha}(f) &:= \mathcal W(f) - \alpha \bigPi^1(f). \\
\end{split}
\end{equation}

\end{Def}

A crucial quantity to be investigated is the following
\begin{Def}\label{alphab}
Let $\;\beta^b\;$ be the $\;\bigPi^2$-Lagrange multiplier of the homogenous torus $\;f^b.\;$ Then we define

$$\alpha^b:=\text{ max }\{\alpha \ | \ \  \delta^2 \mathcal W_{\alpha, \beta^b} \geq0\}.$$

\end{Def}

With these notations the following Corollary is obtained as a further byproduct of the arguments proving the Theorem. \\

\begin{Cor}
For every $\;b \sim 1\;$ fixed there exists $\;a^b>0\;$ small such that for all $\;\alpha < \alpha^b\;$ the minimization problem

 \begin{equation}
\begin{split}
Min_{b} := \inf \{ \  &\mathcal W_\alpha(f)  |\   f: T^2_{b}:= \C/(2\pi \Z + 2\pi b i \Z) \longrightarrow S^3 \text{ smooth immersion with }\\
  &0 \leq \bigPi^1(f)  \leq a^b \text{ and }\; \bigPi^2(f) = b\ \} 
 \end{split}
 \end{equation}
 
is attained at the homogenous torus $\;f^b$.\\
\end{Cor}
\begin{figure}
\vspace{0.5cm}
\includegraphics[width= 0.3\textwidth]{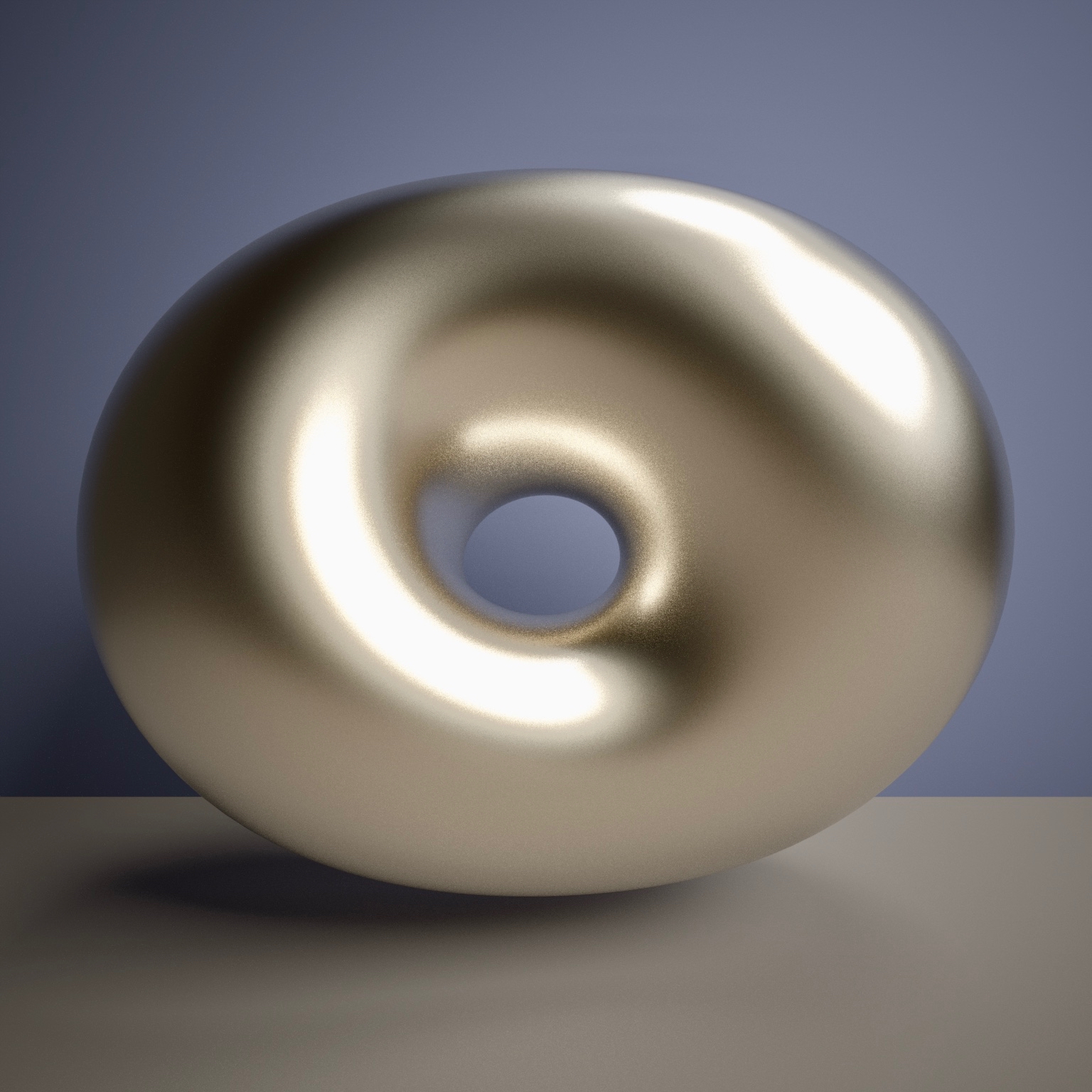}\hspace{1.5cm}
\includegraphics[width= 0.3\textwidth]{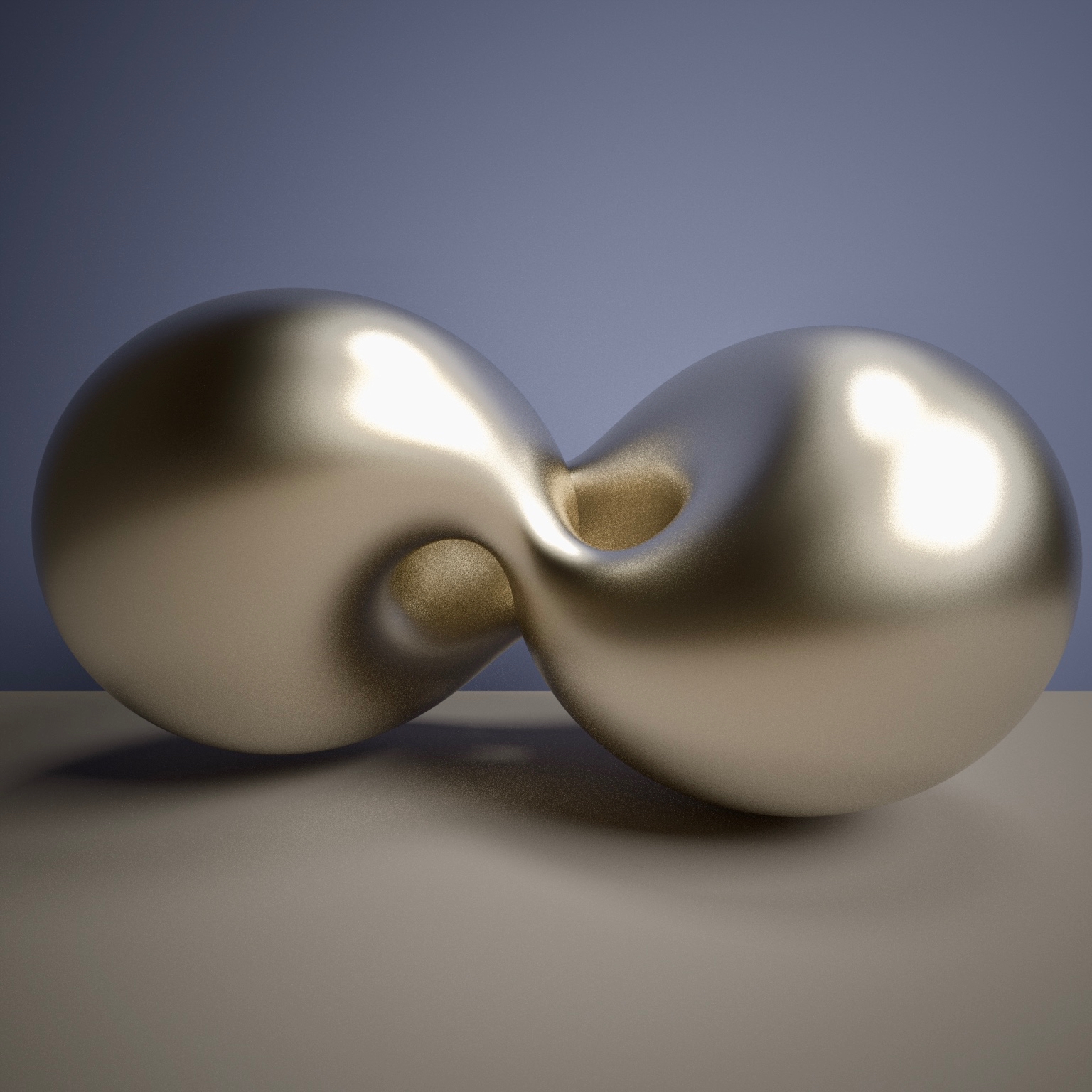}
\caption{ 
Two $(1,2)$-equivariant constrained Willmore tori (with intrinsic period $1$). The tori lie in a $2$-parameter family of surfaces deforming the Clifford torus. This family minimizes the Willmore functional in the respective conformal classes for surfaces "close enough" to the Clifford torus.  Images by Nick Schmitt.}
\label{1lobe}
\end{figure}

The above Theorem and Corollary extends the results in \cite{NdiayeSchätzle1}  which states that the homogenous tori minimizes the Willmore energy in their respective rectangular conformal class in a neighborhood of the square one. The main  difference between \cite{NdiayeSchätzle1} and our case here is that homogenous tori as isothermic surfaces are degenerate w.r.t. to the projection to Teichm\"uller space. Thus by relaxing the minimization problem, Ndiaye and Sch\"atzle were able to restrict to a space where isothermic surfaces solve the relaxed Euler-Lagrange equation and become non-degenerate w.r.t. the associated constraint. Hence they could use the existence and regularity result of \cite{KuwertSchätzle} and the compactness result of \cite{NdiayeSchätzle1} to obtain a  family of abstract minimizers of the constrained Willmore problem smoothly close to the Clifford torus. Furthermore, they show that smoothly close to the Clifford torus there exist only one unique $1$-dimensional family of constrained Willmore tori which are also critical with respect to the relaxed problem using the implicit function theorem.  Therefore the abstract minimizers must coincide with the family of homogenous tori. \\

This is in stark contrast to the case of non-rectangular conformal types. In fact, while  the unique family of constrained Willmore minimizers obtained in \cite{NdiayeSchätzle1} consists of isothermic surfaces, candidates surfaces with non-rectangular class are necessarily non-isothermic, see \cite{HelNdi2}. Further, it is well known within the integrable systems community that there exist various families of constrained Willmore tori deforming\footnote{By deforming a surface $\;f\;$ we mean a smooth family of surfaces containing $\;f.\;$} the Clifford torus covering the same conformal types, as also discussed in \cite{HelNdi2}.  

\begin{figure}[h!]
\includegraphics[width= 0.23\textwidth]{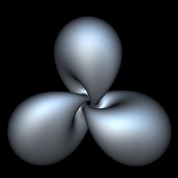} \hspace{0.1cm}
\includegraphics[width= 0.23\textwidth]{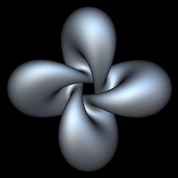} \hspace{0.1cm}
\includegraphics[width= 0.23\textwidth]{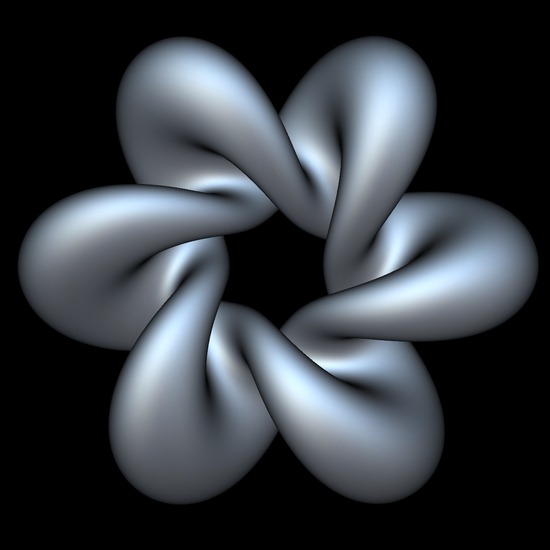} \hspace{0.1cm}
\includegraphics[width= 0.23\textwidth]{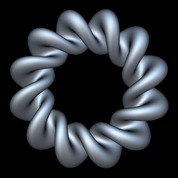} 

\caption{ Equivariant Willmore tori constructed by Ferus and Pedit \cite{FerusPedit}. Each of these surfaces lie in a $1$-parameter family deforming a homogenous torus. Images by Nick Schmitt.
}
\label{1lobe}
\end{figure}

These known families consists of tori given by the preimage of (constrained) elastic curves on $\;S^2\;$ under the Hopf fibration, and are isothermic if and only if they are homogenous \cite{He2}. Moreover, in contrast to tori of revolution, every conformal structure on a torus can be realized by a constrained Willmore Hopf torus \cite{He2}. It has been conjectured by Franz Pedit, Ulrich Pinkall and Martin U. Schmidt that constrained Willmore minimizers should be of Hopf type. Though we disprove this conjecture in this paper,  the actual minimizers we construct lie in the associated family of constrained Willmore Hopf tori, where the Hopf differential of the minimizer is just the one of the associated Hopf surface rotated by a phase. It turns out that the various families deforming the Clifford torus mentioned before can be analytically distinguished by looking at their limit Lagrange multiplier as they converge to the homogenous tori at rectangular  conformal classes. This suggest that to determine the non-rectangular constrained Willmore minimizers we need more control on the abstract minimizers than in the Ndiaye-Sch\"atzle case \cite{NdiayeSchätzle1}, namely the identification of the limit Lagrange multiplier to be exactly $\;\alpha^b\;$ rather than just bounded from above by $\;\alpha^b.$\\

The paper is organized as follows. In the second section we state the main observations leading to a strategy to prove the Main Theorem. It turns out that the degeneracy of an isothermic surface with respect to a penalized Willmore functional (i.e., the second variation has non-trivial kernel) is crucial for the existence of families deforming it. We also observe that the Lagrange multiplier is given by the derivative of Willmore energy with respect to the conformal class. These two properties provide sufficient information to charaterize the {\em possible} limit Lagrange multipliers $\;(\alpha^b, \beta^b)\;$ for a family of constrained Willmore minimizers converging to a homogenous torus $\;f^b,\;$ which we compute in the third section. In the fourth and fifth section we proof the Main Theorem \ref{MainTheorem}. Candidate surfaces $f_{(a,b)}$ parametrized by their conformal class $\;(a,b),\;$ with $\;b \sim 1,\;$ $b \neq1\;$ and $\;a \sim_b 0^+\;$ have been constructed in \cite{HelNdi2} satisfying

\begin{itemize}
\item $f_{(0,b)} = f^b\;$ is homogenous,
\item $f_{(a,b)}\;$ is non degenerate for $\;a \neq0,\;$ and $\;f_{(a,b)} \longrightarrow f^b\;$ smoothly as $\;a \longrightarrow 0$, 
\item for every $\;b \sim 1,\;$ $b \neq 1\;$ fixed and $\;a\neq 0,\;$ the corresponding Lagrange multipliers $\;\alpha_{(a,b)},\;$ and $\;\beta_{(a,b)}\;$ satisfy

$$\alpha_{(a,b)} \nearrow \alpha^b \quad \text{ and } \quad  \beta_{(a,b)}\longrightarrow \beta^b, \quad \text{ as } a \longrightarrow 0.$$

\end{itemize}

This family is in fact real analytic for $\;a >0\;$ and $\;\alpha_{(a,b)}\;$ is shown to be monotonically decreasing in $\;a.\;$ {\footnote{We can assume without loss of generality that $\; a\geq 0.\;$ The choice of a sign correspond to the choice of an orientation on the surface and is equivalent to choosing $\;\delta^2 \bigPi^1(f_{(0,b)}) >0.\;$}} \\

Thus the proof consists of \ref{MainTheorem} consists of two steps
\begin{enumerate}
\item {\bf Classification}: \\
We classify all solutions  $f$ of the constrained Euler-Lagrange equation satisfying
\begin{itemize}
\item $f\;$ is close to a stable\footnote{By stability we mean that $\;\delta^2\mathcal W_{\beta^b}  > 0\;$ up to invariance.} homogenous torus $\;f^b$ ( $b \neq1 $ ) in $\;W^{4,2}\;$  
\item its Lagrange multiplier $\;(\alpha, \beta)\;$ is close to $\;(\alpha^b, \beta^b)\;$ and $\;\alpha \leq \alpha^b$ 
\end{itemize}
via implicit function theorem and bifurcation theory. For $\;b \sim1,\;$ $b \neq 1\;$ fixed  we obtain a unique branch of such solutions $\;f{(a,b)}\;$ parametrized by its conformal type which therefore must coincide with the family of candidate surfaces $\;f_{(a,b)}.$  \\

\item {\bf Global to Local}: \\
We show the existence of constrained Willmore minimizers $\;f^{(a,b)}\;$ with conformal structure $\;(a,b)\;$ with $\;b \sim 1,\;$ $b \neq 1\;$  and $\;a \sim_b0^+\;$ such that their Lagrange multipliers $\;\alpha^{(a,b)}\;$ converge (up to a zero set) to $\;\alpha^b\;$ as $\;a \longrightarrow 0\;$ (and as the surfaces ``converge" to $\;f^b$). Thus  for $\;b\sim 1,\;$ $b \neq 1\;$ fixed  these abstract minimizers can be identified for almost every $\;a\sim_b 0^+\;$ to coincide with the family $\;f{(a,b)} = f_{(a,b)}.\;$ By continuity of the minimal Willmore energy $\;\omega(a,b)\;$ \cite{KuwertSchätzle} (and by the regularity of the candidates) we then obtain that the candidate surfaces $\;f_{(a,b)}\;$ minimize for every $\;a \sim_b 0^+.$
\end{enumerate}

\subsection*{Acknowledgments:} We want to thank Prof. Dr. Franz Pedit and Prof. Dr. Reiner Sch\"atzle for bringing our attention to the topic of this paper and for helpful discussions. We would also like to thank Dr. Nicholas Schmitt for supporting our theoretical work through computer experiments that helped finding the right properties for candidates surfaces to be minimizer and for producing the images used here. Moreover, we thank Dr. Sebastian Heller for helpful discussions.

\section{Strategy and main observations}\label{strategy}

In this section we state key ideas and the strategy for the proof of the Main Theorem (Theorem \ref{MainTheorem}).  We follow the notations used in \cite{KuwertLorenz}. \\

The Teichm\"uller space of tori can be identified with the upper half plane $\;\H^2.\;$ Thus let 

$$  \bigPi(f) = \big(  \bigPi^1(f),   \bigPi^2(f)\big)$$ 

be the projection map of an immersion $\;f: T^2 \longrightarrow S^3\;$ to $\;\H^2\;$ such that the Clifford torus 

$$f^{1} :  T^2_{1} =  \C /\big(\sqrt{2}\pi\Z + \sqrt{2}\pi i \Z\big) \longrightarrow S^3 \subset \C^2$$ 

parametrized by 

$$f^1(x,y) = \frac{1}{\sqrt {2}}\begin{pmatrix} e^{\tfrac{ix}{\sqrt{2}}}, & e^{\tfrac{iy}{\sqrt{2}}} \end{pmatrix}$$

is mapped to $  \bigPi (f^1) = (0,1).$ Then we can write the Euler-Lagrange equation for a constrained Willmore torus as
 
\begin{equation}\label{ELana}\delta \mathcal W = < \omega, \mathring{\text{II}}> = \alpha\delta  \bigPi^1 + \beta \delta  \bigPi^2,
\end{equation}

with Lagrange multipliers $\;\alpha\;$ and $\;\beta.\;$ The surface is non-isothermic if and only if the Lagrange multipliers are uniquely determined (after choosing a base in $\;\H^2$). At the Clifford torus, and more generally, at homogenous tori we have $\;\delta   \bigPi^1 = 0\;$ and thus the $\;\alpha-$Lagrange multiplier can be arbitrarily chosen. As already discussed before, it is well known that there exist various families of (non-isothermic) constrained Willmore tori deforming a homogenous torus. These families can be distinguished by the limit of their $\;\alpha$-Lagrange multiplier as they converge smoothly to the homogenous torus. The obstructions  for such families to exist and how these limit Lagrange multipliers relate to their Willmore energy is summarized in the following Lemma. Though the proof of the Lemma \ref{mainobservation} is trivial, these observations give the main intuition for the dependence of the minimum Willmore energy on the conformal classes. \\

\begin{Lem}[Main observation]\label{mainobservation}$\ $\\
Let $\;\tilde f^{(a,b)}\;$ be a family of smooth immersions with conformal type 

$${(a,b) =: (\tilde a^2, b) \in [0, a^2_0) \times (1-b_0, 1+b_0)}$$ 

for some positive numbers  $\;a_0,\; b_0 \in \R\;$ such that the map

$$(\tilde a,b) \longmapsto \tilde f^{(a,b)} \in C^2\left(\left[0, a_0\right) \times (1-b_0, 1+b_0), W^{4,2}\right),$$

and  $\;\delta  \bigPi^1 \big({\tilde f^{( a,b)}}\big ) = 0,\;$ but $\;\delta   \bigPi^1 \big({\tilde f^{( a,b)}}\big ) \neq 0\;$ for $\;a \neq 0.\;$  Further, let $\;\tilde \alpha( a,b)\;$ and $\;\tilde \beta( a,b)\;$ be the corresponding Lagrange multipliers  satisfying 

$$(\tilde a,b) \longmapsto \tilde \alpha(a,b), \tilde \beta(a,b) \in C^2\left([0, a_0) \times (1-b_0, 1+b_0), W^{4,2}\right),$$

and 
$\;\tilde \omega( a, b):= \mathcal W\big({\tilde f^{( a,b)}}\big ).\;$ Then we obtain

\begin{enumerate}
\item$$\frac{\partial \tilde \omega( a,b)}{\partial a} = \tilde \alpha(a,b) \text{ for } a \neq 0 \quad \text{ and } \quad \lim_{\tilde a\rightarrow 0} \frac{\partial \tilde \omega(a,b)}{\partial  a}  = \tilde\alpha(0,b) =: \tilde\alpha^b \quad \forall b,$$

\item$$\frac{\partial \tilde \omega( a,b)}{\partial b} = \tilde\beta( a,b) \text{ for } a\neq 0 \quad \text{ and } \quad \lim_{ a\rightarrow 0} \frac{\partial \tilde \omega( a,b)}{\partial b}  = \tilde\beta(0,b)  =: \tilde \beta^b\quad \forall b,$$

\item
$\varphi^b:= \del_{\tilde a} f^{( a, b)}|_{a=0}$ satisfies 

$$ \delta^2\left(\mathcal W_{\tilde \alpha^b, \tilde \beta^b} \right)\big({\tilde f^{( 0,b)}}\big )(\varphi^b, \varphi^b) = 0 \quad \forall b.$$

\end{enumerate}
\end{Lem}

\begin{proof} The proof only uses the definition of the family, the constrained Euler-Lagrange equation and its derivatives. By assumption we have that $\;\del^k_{\tilde a}\del_b^l \tilde f^{(a, b)}\;$ exist and is continuous on 
$$[0, a_0) \times (1-b_0, 1+b_0) \quad \text{for}\quad k,l = 0,1,2.$$ 
Since $\;\del_{\tilde a} = 2\sqrt{a}\del_a\;$ for $\;a \neq 0\;$ we have that $\;\del_{ a} \tilde f^{( a, b)}\;$ exist for $\;a \neq 0\;$ but $\;\lim_{a\rightarrow 0}\del_{ a} \tilde f^{( a, b)}\;$ cannot exist due to the degeneracy of $\;f^{(0,b)}.$
\begin{enumerate}
\item Let $\;\varphi:= \frac{\partial \tilde f^{(a,b)}}{\partial  a}$ for $a\neq 0.\;$ 
Then $\;\frac{\partial \tilde \omega( a,b)}{\partial a} = \delta \mathcal W\big({\tilde f^{(a,b)}}\big ) (\varphi)\;$ for $\;a\neq 0\;$ and hence by the constrained Euler-Lagrange equation we have:

$$\frac{\partial \tilde \omega( a,b)}{\partial  a}= \tilde \alpha ( a,b) \delta \bigPi^1 \big({\tilde f^{( a,b)}}\big )(\varphi) + \tilde \beta( a,b) \delta \bigPi^2 \big({\tilde f^{( a,b)}}\big )(\varphi), \quad \text{ for } a\neq 0.$$
\noindent
Since  $\;\bigPi(\tilde f^{(a,b)}) = ( a,b),\;$ we obtain for $\; a\neq 0\;$ that 

\begin{equation}\label{Piab}
\delta \bigPi^1 \big({\tilde f^{( a,b)}}\big )(\varphi) = 1 \quad \text{and}\quad \delta \bigPi^2 \big({\tilde f^{( a,b)}}\big )(\varphi) = 0
\end{equation}

 and therefore 
 
 \begin{equation*}
\frac{\partial \tilde \omega( a,b)}{\partial  a} = \tilde \alpha(a,b), \quad a \neq 0.
\end{equation*}

\noindent
Passing to the limit gives the first assertion.\\
\item This follows completely analogously to (1).\\
\item In this case we test the Euler-Lagrange equation by $\;\varphi\;$ and obtain for $\;a \neq 0\;$

$$ \delta \mathcal W \big({\tilde f^{( a,b)}}\big ) (\varphi) = \tilde \alpha ( a,b) \delta \bigPi^1 \big({\tilde f^{( a,b)}}\big )(\varphi) + \tilde \beta(a,b) \delta \bigPi^2 \big({\tilde f^{( a,b)}}\big )(\varphi).$$

\noindent
Now differentiating this equation with respect to $\; a\;$ yields

\begin{equation*}
\begin{split}
 \delta^2 \mathcal W\big({\tilde f^{( a,b)}}\big )(\varphi, \varphi) &= \tilde \alpha (a,b) \delta^2 \bigPi^1\big({\tilde f^{( a,b)}}\big )(\varphi, \varphi) + \tilde \beta( a,b) \delta^2 \bigPi^2\big({\tilde f^{( a,b)}}\big )(\varphi, \varphi)\\ &+ \frac{\partial \tilde \alpha( a, b)}{\partial  a} \delta \bigPi^1\big({\tilde f^{( a,b)}}\big )(\varphi) + \frac{\partial \tilde \beta( a, b)}{\partial  a} \delta \bigPi^2\big({\tilde f^{( a,b)}}\big )(\varphi) \quad \text{ for } a\neq 0.
 \end{split}
 \end{equation*}

\noindent
In order to pass to the limit, it is necessary to replace $\;\varphi\;$ by $\;\sqrt{ a} \varphi.\;$ This gives

\begin{equation*}
\begin{split}
 \delta^2 \mathcal W\big({\tilde f^{( a,b)}}\big )(\sqrt{a}\varphi, \sqrt{a}\varphi) &= \tilde\alpha (a,b) \delta^2 \bigPi^1\big({\tilde f^{( a,b)}}\big )(\sqrt{a}\varphi, \sqrt{a}\varphi)\\&+\tilde \beta(a,b) \delta^2 \bigPi^2\big({\tilde f^{( a,b)}}\big )(\sqrt{a}\varphi, \sqrt{a}\varphi)\\ &+ \sqrt{a}\frac{\partial \tilde \alpha(a, b)}{\partial a} \delta \bigPi^1\big({\tilde f^{( a,b)}}\big )(\sqrt{a}\varphi) + \sqrt{a}\frac{\partial \tilde \beta(a, b)}{\partial a} \delta \bigPi^2\big({\tilde f^{( a,b)}}\big )(\sqrt{a}\varphi) .
 \end{split}
 \end{equation*}
 
 \noindent
By assumption we have

$$\lim_{\tilde a \rightarrow 0}2\sqrt{a}\frac{\partial \tilde \alpha(a, b)}{\partial a} = \lim_{\tilde a \rightarrow 0}\frac{\partial \tilde \alpha(a, b)}{\partial \tilde a} \text{ and  }\lim_{\tilde a \rightarrow 0}2\sqrt{a}\frac{\partial \tilde \beta(a, b)}{\partial a} = \lim_{\tilde a \rightarrow 0}\frac{\partial \tilde \beta(a, b)}{\partial \tilde a}$$

\noindent
exist and moreover,  $\;\lim_{\tilde a\rightarrow 0}\delta \bigPi^1\big({\tilde f^{( a,b)}}\big )(\sqrt{a} \varphi) = 0\;$ and $\;\delta \bigPi^2 \big({\tilde f^{( a,b)}}\big )(\sqrt{a}\varphi) = 0\;$ as in \eqref{Piab}. Therefore we obtain for $\;\tilde a\longrightarrow 0$

$$\delta^2 \left(\mathcal W_{\tilde \alpha^b, \tilde \beta^b}\right )\big({\tilde f^{( 0,b)}}\big ) (\varphi^b, \varphi^b) = 0.$$

\end{enumerate}
\end{proof}

\begin{Rem}
For any family $\;\tilde f^{(a,b)}\;$ the quantities used and computed in the above lemma only depends the normal part of the variation $\;\varphi\;$ and $\;\varphi^b.\;$ We will denote these normal variations again by $\;\varphi\;$ and $\;\varphi^b\;$ in the following. 
\end{Rem}

The first assertion of the lemma states that for a any family of constrained Willmore tori $\;\tilde f^{(a,b)},\;$ with the properties as in the Lemma, their Lagrange multipliers corresponds to the derivative of the Willmore energy $\;\tilde \omega(a,b).\;$ At $\;a=0\;$ and for $\;b \sim 1\;$ fixed we have by \cite{NdiayeSchätzle1} that the homogenous torus $\;f^b\;$ is the unique constrained Willmore minimizer. This suggests that the Lagrange multipliers $\;\tilde \alpha(a,b)\;$ of a family $\;\tilde f^{(a,b)}\;$ of putative constrained Willmore minimizers with $\;\tilde f^{(0,b)} = f^b\;$ should have the the smallest possible limit $\;\tilde \alpha^b\;$ as $\;a \longrightarrow 0.\;$ A necessary (and as we will later see a sufficient) condition for such a family to exist is given by the second statement of Lemma \ref{mainobservation}, namely the degeneracy of the second variation of the penalized Willmore functional $\;\mathcal W_{\tilde \alpha^b, \beta^b}.$ \\

\begin{Rem}\label{betab}
The limit Lagrange multiplier $\;\beta^b \;$ is uniquely determined as the $\;\beta$-Lagrange multiplier of the homogenous torus $\;f^b\;$ due to the non-degeneracy of the $\;\bigPi^2$-direction.
The discussion above suggest that the first step towards the proof of the main Theorem, Theorem \ref{MainTheorem}, is to determine

$$\alpha^b=\text{ max }\big\{ \;\alpha \ | \ \  \delta^2 \mathcal W_{\alpha, \beta^b}\big(f^b\big) \geq0 \; \big\}.$$
\end{Rem}

\noindent
It is well known that the Clifford torus, and thus all homogenous tori smoothly close to the Clifford torus, is strictly stable (up to invariance). Therefore $\;\alpha^b\;$ is strictly positive by fixing an orientation, i.e., $\;\bigPi^1(f)\geq 0.\;$ We will compute in the next section that it is also finite. Further, since we show in Proposition \ref{1d} that the kernel of $\;\delta^2\mathcal W_{\alpha^b,\beta^b}\big(f^b\big)\;$ is $1$-dimensional for $\;b \sim 1\;$ and $\;b \neq 1\;$ (up to invariance), the third statement of Lemma \ref{mainobservation} implies that this kernel determines the normal variation of the candidate family $\;f_{(a,b)}\;$ up to reparametrization. Moreover, the normal variation $\;\varphi^b\in \delta^2 \mathcal W_{\alpha^b, \beta^b}\big({f^b}\big)\;$ for $\;b \neq1\;$ is computed to have (intrinsic) period one and independent of the $y$-direction (see Section \ref{stability}) of a reparametrized homogenous torus. More precisely, for
 
$$T^2_{b}:= \C /\left (2\pi \Z \oplus 2 \pi \frac{r^2+ 2irs}{r^2+4s^2} \Z \right) $$

we consider the homogenous torus $\;f^b\;$ parametrized as an $(1,2)$-equivariant surface\footnote{Equivariant surfaces are those with a $1$-parameter family of isometric symmetries, we discuss these surfaces in \cite{He1, HelNdi2}. The $\;T^2_b\;$ used in the definition is biholomorphic to $\;T^2_b  = \C/\left (2 \pi r \Z \oplus 2\pi s i \Z\right).\;$ We state the immersion here with this lattice to emphasize that it is $(1,2)$-parametrized.}

\begin{equation}\label{12fb}
\begin{split}
f^b: \ \ T^2_{b} &\longrightarrow S^3, \\
(x,y) &\longmapsto \begin{pmatrix}r e^{i \left(y+2\tfrac{s}{r}x\right)}, & s e^{i\left (2y-\tfrac{r}{s}x\right)}\end{pmatrix}.
\end{split}
\end{equation}

The independence of $\;\varphi^b\;$ w.r.t. the $y$-direction means that  the corresponding family $\; f_{(a,b)}\;$ (with the properties of Lemma \ref{mainobservation}) are infinitesimally $(1,2)$-equivariant. Furthermore, in our case knowing the limit Lagrange multiplier $\;\alpha^b\;$ is tantamount to knowing the normal variation $\;\varphi^b,\;$ since $\;\mathcal W_{\alpha^b, \beta^b}\;$ is linear in $\;\alpha^b$. \\

For $\;\alpha\in [0,\alpha^b)\;$ the second variation $\;\delta^2 \mathcal W_{\alpha, \beta^b}(f^b)\;$ is strictly positive (up to invariance), thus $2$-dimensional families deforming the homogenous tori smoothly with

 $$\lim_{a\rightarrow0} \alpha(a,b) = \alpha$$ 
 
 cannot exist. Indeed, the following Lemma shows that this is even true in $\;W^{4,2}$-topology. It can be proven by using the same arguments as in \cite{NdiayeSchätzle1}.\\

\begin{Lem}\label{alpha<alpha^b}For $\;b \sim 1\;$ fixed and $\;\alpha^b\;$ defined as in Definition \ref{alphab} let $\;\alpha \in \R \;$ with $\;\alpha<  \alpha^b.\;$ Then the homogenous tori $\;f^b\;$ is the unique solution (up to invariance) of the equation 

$$\delta \mathcal W (f) = \alpha_f \delta \bigPi^1(f) + \beta^b\delta \bigPi^2(f)$$

with $\;\alpha_f \sim \alpha\;$ and $\;f \sim f^b\;$ in $vW^{4,2}$, $\;\bigPi^1(f) \geq 0\;$ and $\;\bigPi^2(f) = b.$
 
\end{Lem}

At $\;\alpha= \alpha^b\;$ (and $\;b \sim1\;$, $b \neq1$) the situation is very different. Using Integrable Systems Theory we can construct a  family of $(1,2)$-equivariant constrained Willmore tori $\;f_{(a,b)}\;$ parametrized by their conformal type $\;(a, b) \sim (0,b)\;$ deforming smoothly the homogenous torus $\;f^b = f_{(0,b)}\;$ such that the corresponding Lagrange multipliers $\;\alpha_{(a,b)} \nearrow \alpha^b\;$ converge from below as $\;a \longrightarrow 0.\;$ In fact, we prove even more in \cite{HelNdi2}.\\

\begin{The}[\cite{HelNdi2}]
For $\;b = \tfrac{s}{r} \sim 1,\;$ with $\;r^2 + s^2 = 1\;$ and $\;b \neq 1\;$ fixed there exists for $\;a \sim_b 0^+\;$ a family of $(1,2)$-equivariant constrained Willmore immersions

 $$f_{(a,b)} :  T^2_{(a,b)}:= \C /2\pi r\big( \Z \oplus (a + ib) \Z\big) \longrightarrow S^3$$
 
 such that 
 
 $$(\sqrt{a}, b) \longmapsto f_{(a,b)} \in C^\omega \Big((0, a_0)\times[1, 1+ b_0), C^\infty_{\text{Imm}} \Big ) \cap C^2 \Big([0, a_0]\times[1, 1+ b_0), C^\infty_{\text{Imm}}\Big),$$

 where $\;C^\infty_{\text{Imm}}\; $ is the space of smooth immersions from a torus into $\;S^3\;$ and $\;C^\omega\;$ denote the space of real analytic maps. Moreover,
 
 $$(a, b)\longmapsto \mathcal W\big(f_{(a,b)}\big)  \in C^2 \Big([0, a_0]\times[1, 1+ b_0)\Big)$$
 
satisfy the following 

\begin{enumerate} 

\item $\text{For all } \;b\sim1, b\neq1 \text{ fixed}, \  f_{(a,b)}\;$ converge smoothly to the homogenous torus $\;f^b\;$ as $\;a \longrightarrow 0\;$ given by

$$f^b: T^2_{b} = T^2_{(0,b)} \longrightarrow S^3, \quad (x,y) \longmapsto \begin{pmatrix}r e^{i\left(2y+\tfrac{s}{r}x\right)}, & s e^{i\left(y-2\tfrac{r}{s}x\right)}\end{pmatrix}.$$ \\

\item The immersions $\;f_{(a,b)}\;$ are non-degenerate for $\;a \neq0\;$  and satisfy

 $$\delta \mathcal W(f_{(a,b)}) = \alpha_{(a,b)}\delta \bigPi^1 + \beta_{(a,b)}\delta \bigPi^2 \quad \text{ for } a \neq 0$$
 
 \noindent
with Lagrange multipliers $\;\big(\alpha_{(a,b)}, \beta_{(a,b)}\big)\;$ such that $\;\alpha_{(a,b)}\nearrow \alpha^b\;$ monotonically and $ \;\beta_{(a,b)} \longrightarrow \beta^b\;$ as $\;a \longrightarrow 0.$ \\

\end{enumerate}

\end{The}

\begin{Rem}
The candidates are constructed as conformal immersions from $\;T^2_{(a,b)}\;$ to $\;S^3.\;$ Since $\;T^2_{(a,b)}\;$ is $\;C^\infty$-diffeomorphic to $\;T^2_b,\;$ the space $\;C^\infty_{Imm}\big(T^2_{(a,b)}\big) \;$ is canonically isomorphic to $\; C^\infty_{Imm}\big(T^2_b\big),\;$ i.e., it does not depend on the conformal type of the domain.  By the convergence of $\;f_{(a,b)}\;$ to $\;f^b\;$ we mean the convergence of the maps under this identification $\;C^\infty_{Imm}\big(T^2_{(a,b)}\big) \cong C^\infty_{Imm}\big(T^2_b\big).$\\
\end{Rem}

\begin{Rem}\label{ab}
By Lemma \ref{mainobservation} we obtain that

$$\left(\partial_{\sqrt{a}} f_{(a,b)}|_{a= 0} \right)^\perp =: \varphi^b \in \text{Ker}\left(\delta^2 \mathcal W_{\alpha^b, \beta^b}\right)\big(f^b\big).$$

Moreover Lemma \ref{mainobservation} also implies that for $\;b \sim 1,\;$ $b \neq 1\;$ fixed, the map $\;a \longmapsto \mathcal W \big(f_{(a,b)}\big)\;$ is monotonically increasing and concave in $\;a \sim 0^+.\;$ Hence there exist $\;a^b>0\;$ and small such that for  all $\;a \in [0, a^b)$

\begin{equation}\label{fpmca}
\mathcal W_{\alpha^b}\big(f_{(a,b)}\big) < \mathcal W_{\alpha^b}\big(f^b \big).
\end{equation}

This means that  the homogenous tori $\;f^b\;$ cannot be the minimizer of $\;\mathcal W_{\alpha^b}\;$ among immersions $\;f\;$ with $\;0\leq \bigPi^1(f)\leq a^b\;$ and $\;\bigPi^2(f) = b$. 

\end{Rem}

At $\;f^b\;$ the second variation of $\;\mathcal W_{\alpha^b, \beta^b}\;$ is degenerate. Thus a simple application of the implicit function theorem as in \cite{NdiayeSchätzle1, NdiayeSchätzle2} to classify all solutions close to $\;f^b\;$ in $\;W^{4,2}\;$ is not possible. Instead, we use bifurcation theory from simple eigenvalues for the classification. For this we first show in Proposition \ref{1d} that the kernel of $\;\delta^2 \mathcal W_{\alpha^b, \beta^b}(f^b),\;$ for $\;b\neq 1,\;$ is only $1$-dimensional up to invariance. Then together with Lemma \ref{n-thderivative}  the following classification result is proven:\\

\begin{The}\label{classification}For $\;b \sim 1,\;$ $b\neq 1\;$ fixed and up to taking $\;a^b\;$ of Remark \ref{ab} smaller,  there exists  (up to invariance) a unique family of non-degenerate solutions $\;f(a,b)\;$ for $\;a \neq0\;$ to the constrained Euler-Lagrange equation \eqref{ELana} parametrized by their conformal type $\;(a,b)\;$ with $\;a \in [0, a^b),\;$ $f{(a,b)} \sim f^b\;$ in $\;W^{4,2}\;$ as $\;a \sim 0^+\;$ and $\;f{(0,b)} = f^b\;$ 

with its Lagrange multipliers $\;\alpha(a,b)\;$ and $\;\beta(a,b)\;$
 satisfying 
 
$$\alpha(a,b)  \nearrow \alpha^b \quad  \text{and} \quad \beta(a,b) \longrightarrow \beta^b \quad \text{as} \quad a \longrightarrow 0.$$

In particular, the only solution $\;f\;$ of the constrained Willmore equation with conformal type $\;\bigPi(f) = (0, b)$, $\;\alpha = \alpha^b\;$ and $\;\beta=\beta^b\;$ is the homogenous torus $\;f^b$.\\
\end{The} 

Since our candidate surfaces from Theorem \ref{explicitcandidates} has Lagrange multiplier $\;\alpha_{(a,b)} \nearrow \alpha^b\;$ and smoothly converge to $\;f^b\;$ as $\;a \longrightarrow 0\;$ we can conclude that $\;f_{(a,b)}= f{(a,b)}\;$ for all $\;a \in [0, a^b)\;$ and $\;b \sim 1,\;$ $b \neq1.$\\

To prove the main Theorem (Theorem \ref{MainTheorem}) it remains to show that there are abstract minimizers $\;f^{(a,b)}\;$ of the constrained Willmore problem for the conformal class $\;(a,b)\;$ with $\;b \sim 1\;$ and $\;a \in [0, a^b),\;$ which clearly exist by \cite{KuwertSchätzle},  satisfying the additional property that their Lagrange multipliers $\;\alpha^{(a,b)}\longrightarrow \alpha^b,\;$ $\beta^{(a,b)} \longrightarrow \beta^b\;$ and $\;f^{(a,b)} \sim f^b\;$ in $\;W^{4,2}\;$ as $\;a \longrightarrow 0.\;$ Then these abstract minimizers would be covered by the classification result given by Theorem \ref{classification}, and must therefore coincide with $\;f{(a,b)}\;$ and the candidate surfaces $\;f_{(a,b)}$. \\

\begin{Rem}
Due to technicalities we actually only show the convergence of the Lagrange multipliers $\;\alpha^{(a,b)} \longrightarrow \alpha^b\;$ for $\;a \longrightarrow 0\;$  almost everywhere (and $\;b \sim 1\;$ fixed). More precisely, we show that $\;\alpha(a,b) \longrightarrow \alpha^b\;$ for $\;a\longrightarrow 0\;$ and $\;a \in [0, a^b) \setminus A\;$ for a suitable zero set $\;A.\;$ From this we can conclude that the abstract minimizers $\;f^{(a,b)}\;$ coincide for almost every $\;a\in [0,a^b) \;$ with the candidates surfaces $\;f_{(a,b)}.\;$ Then by the continuity of the minimal energy $\;\omega(a,b)\;$ as shown in \cite{KuwertSchätzle} (and real analyticity of $\;f_{(a,b)}\;$ for $\;a \neq0\;$) we  obtain that $\;f_{(a,b)}\;$ are constrained Willmore minimizers  for every $\;a \in [0, a^b).$
\end{Rem}
The properties of the abstract minimizers are shown by considering a relaxed minimization problem for a penalized Willmore functional as in the following theorem.\\

\begin{The}\label{abstractmin}
For $\;b \sim 1\;$ fixed and up to taking $a^b$ smaller we have that for all $\;a \in [0, a^b)\;$ the minimization problem
\begin{equation}
\begin{split}
\Min_{(a,b)} := \inf \big\{ \mathcal W_{\alpha^b}(f)\; |\   &f: T^2_{b} \longrightarrow S^3 \text{ smooth immersion with }\\
  &0 \leq \bigPi^1(f)  \leq a \; \text{ and }\; \bigPi^2(f) = b\big\}
\end{split}
\end{equation}

is attained by a smooth and non-degenerate (for $\;a \neq 0\;$) constrained Willmore immersion 

$$f^{(a,b)}: T^2_b \longrightarrow S^3$$

 of conformal type $\;(a,b)\;$
with Lagrange multipliers  $\;\alpha^{(a,b)} \nearrow \alpha^b,\;$ $ \beta^{(a,b)} \longrightarrow \beta^b\;$  for almost every $\;a \longrightarrow 0\;$  and  $\;f^{(a,b)} \longrightarrow f^b\;$ in $\;W^{4,2}\;$ for almost every $\;a \in [0, a^b).$ \\\end{The}

The minimizers with respect to the penalized functional $\;\mathcal W_{\alpha^b}\;$ automatically minimize the plain constrained Willmore problem. We briefly discuss the main ingredients for the proof of Theorem \ref{abstractmin}: By the work of Kuwert and Sch\"atzle \cite{KuwertSchätzle2} and Sch\"atzle \cite{Schätzle} we obtain the existence of the minimizers $\;f^{(a,b)}.\;$ Because of Equation \eqref{fpmca} and the classification (Theorem \ref{classification}), these minimizers are always attained at the boundary, i.e.,  $\;\bigPi^1\big(f^{(a,b)}\big) = a.\;$ This together with the relaxation of our constraint imply that $\;\Min_{(a,b)}\;$ is monotonic. Due to this monotonicity of $\;\Min_{(a,b)}\;$ we obtain that the minimal Willmore energy $\;\omega(a,b)\;$ is almost everywhere differentiable with respect to $\;a.\;$ In a second step we show that $\;\del_a\omega(a,b),\;$ where it exists, corresponds to $\;\alpha^{(a,b)}\;$ by constructing a smooth family of surfaces $\;\bar f^{(a,b)}\;$ whose Willmore energy approximates $\;\omega(a,b)\;$ at $\;a_0\;$ up to second order. By the monotonicity of $\;\text{Min}_{(a, b)}\;$ and Lemma \ref{alpha<alpha^b}  we show 
$$\alpha_{sup} := \limsup_{a \rightarrow 0, a.e. } \alpha^{(a,b)}=\alpha^b.$$ 
For $$\alpha_{inf}:=\liminf_{a \rightarrow 0, a.e.} \alpha^{(a,b)}$$
we use again the family $\;\bar f^{(a,b)}\;$ to show that $$\alpha_{inf} \geq 0. $$ Then by Lemma \ref{alpha<alpha^b} we show $$\alpha_{inf} \geq \alpha^b.$$ 

The remaining convergence of $\;\beta^{(a,b)} \longrightarrow \beta^b\;$ and $\;f^{(a,b)} \longrightarrow f^b\;$ in $\;W^{4,2}\;$ for almost every $\;a \longrightarrow 0\;$ follow from \cite{NdiayeSchätzle1}.\\

\section{Stability properties of a penalized Willmore energy}\label{stability}
In the computations below we mostly follow \cite{KuwertLorenz} and thus we refer to that paper for details. To fix the notations, we consider immersions

$$f: T^2=\C / \Gamma \longrightarrow (S^3, \;g_{S^3}),$$

where $\;\Gamma\;$ is a lattice and $\;g_{S^3}\;$ is the round metric on $\;S^3.\;$ Let Imm$\;(\C / \Gamma)\;$ denote the space of all such immersions and $\;$Met$(\C / \Gamma)\;$ the space of all metrics on the torus $\;T^2.\;$ Moreover, let 

\begin{equation*}
\begin{split}G: \text{Imm }(T^2) &\longrightarrow  \text{Met }(T^2)\\
 f &\longmapsto f^*g_{S^3}
\end{split}
\end{equation*}

 be the map which assigns to every immersion its induced metric.  We denote by $\;\pi\;$ the projection from the space of metrics to the Teichm\"ulller space, which we model by the upper half plane $\;\H^2\;$ and with the notations above we can define $\;\bigPi\;$ to be:
 
 $$\bigPi = \pi \circ G : \text{Imm }(T^2) \longrightarrow \H^2.$$ 

As in \cite{KuwertLorenz} we parametrize the homogenous torus with conformal class $\;b= \tfrac{s}{r},\;$ and $\;r^2 + s^2 = 1\;$ as

\begin{equation}\label{parametrizationofhomogenous}
\begin{split}
 f^b: T^2_b &\longrightarrow S^3,\\
 (x,y) &\longmapsto\begin{pmatrix}r e^{i\tfrac{x}{r}}, &s e^{i\tfrac{y}{s}}\end{pmatrix}.
 \end{split}
\end{equation}

We want to compute the value of $\;\alpha^b\;$ which we recall to be

 $$\alpha^b=\text{ max }\big\{\;\alpha \;  | \  \delta^2 \mathcal W_{\alpha, \beta^b}\big(f^b\big) \geq0\;\big\}.$$
 
From \cite{KuwertLorenz} we can derive that $\;\alpha^b\;$ is characterized by the fact that $\;\delta^2 \mathcal W_{\alpha^b, \beta^b}|_{f^b}\geq 0\;$ and there exist a non-trivial normal variation $\;\varphi^b\;$ of $\;f^b\;$ such that 

$$ \delta^2 \mathcal W_{\alpha^b, \beta^b}\big(f^b\big)\big(\varphi^b, \varphi^b\big)= 0, \quad \text{and}\quad \delta^2 \mathcal W_{\alpha, \beta^b}\big(f^b\big)\big(\varphi^b, \varphi^b\big)<0, \text{ for } \alpha >\alpha^b.$$  

We will show that for $\;b \neq 1\;$ the variation $\;\varphi^b\;$ is unique up to scaling, isometry of the ambient space and reparametrization of the surface $\;f^b.\;$ We will also choose the orientation of $\;f^b\;$ and the variation $\;\varphi^b\;$ such that $\;\delta^2 \bigPi^1\big(f^b\big) \geq 0.$ \\

While for $\;b= 1\;$ the exact value of $\;\alpha^1\;$ and the associated normal variations can be computed, $\;\alpha^b\;$ for $\;b \neq 1\;$ does not have a nice explicit form. Nevertheless, we will show that the unique normal variation $\;\varphi^b\;$ characterizing $\;\alpha^b\;$ remain the same (in a appropriate sense) for all $\;b\sim1.\;$  In fact, the normal variation $\;\lim_{a\rightarrow 0} \big(\del_a f_{(a,b)}\big)^\perp\;$ is the information we use to show that the Lagrange multipliers of our candidates $\;f_{(a,b)}\;$ converge to the $\;\alpha^b\;$ as $\;a \longrightarrow 0,\;$ see Section \ref{convergence}.\\

We first restrict to the case $\;b=1$ -- the Clifford torus. Since $\;\beta^1 = 0\;$ 
we investigate for which $\;\alpha\;$ the Clifford torus $f^1$ is stable for the penalized Willmore functional $\;\mathcal W_{\alpha} = \mathcal W - \alpha\bigPi^1.$

The second variation of the Willmore functional is well known. Thus we first concentrate on the computation of the second variation of $\;\bigPi^1.\;$ Another well known fact is $\;\delta \bigPi^1(f^1) = 0.\;$ Moreover, we have 

$$D^2 \bigPi^1\big(f^1\big)\big(\varphi, \varphi\big) = D^2\pi^1\Big(G\big(f^1\big)\Big)\Big(DG\big(f^1\big)\varphi, DG\big(f^1\big)\varphi\Big) +  D\pi^1\Big(G\big(f^1\big)\Big)\Big(D^2G\big(f^1\big)\big(\varphi,\varphi\big)\Big)$$

The first term is computed in Lemma $4$ of \cite{KuwertLorenz} to be

$$D \pi^1\Big(G\big(f^1\big)\Big)\Big(D^2G\big(f^1\big)\big(\varphi,\varphi\big)\Big) = - \tfrac{1}{\pi^2} \int_{T^2_1} <\nabla^2_{12} \varphi, \varphi> d \mu_{g_{euc}},$$

for normal variations $\varphi.$ It remains to compute the second term

$$D^2\pi^1\Big(G\big(f^1\big)\Big)\Big(DG\big(f^1\big)\varphi, DG\big(f^1\big)\varphi\Big).$$

By a straight forward computation (or by Lemma $2$ of \cite{KuwertLorenz}) we have  

$$D G\big(f^1\big)\varphi = - 2 \int_{T^2_1}<\text{II}, \varphi> d \mu_{g_{euc}},$$

where $\;\text{II}\;$ is the second fundamental form of the Clifford torus, which is trace free.\\

\noindent
Let $\;u\;$ and $\;v \in S_2(T^2_1)\;$ be symmetric $2-$forms satisfying 

$$\Tr_{euc} u = \Tr_{euc}v = 0 \quad \text{ and } \quad v \perp_{euc} S_2^{TT}(g_{euc}),$$

where $\;S_2^{TT}(g_{euc})\;$ is the space of symmetric, covariant, transverse traceless $2$-tensors with standard basis $\;q^1\;$ and $\;q^2\;$ and $\;q^i(t)\;$ the corresponding basis of $\;g(t).\;$ Let $\;g(t) = g_{euc} + t u\;$ and $\;q^i(t) = q^i\big(g(t)\big)\;$ such that $\;\big(q^i(t) - q^i\big) \perp_{euc} S_2^{TT}(g_{euc}).\;$ Then we can expand $\;v\;$ by 

$$v= v_i(t) q^i(t) + v^\perp(t), \quad \text{where }v^\perp(t)\perp_{euc} S_2^{TT}(g_{euc}).$$

By assumption we have $\;v_i(0) = 0\;$ and thus 
$$D^2\pi^1(g_{euc})(u, v) = \frac{d}{dt} D\pi\big(g(t)\big) \cdot v|_{t=0} = v'_{1}(0) D\pi(g_{euc}) \cdot q^1,$$
where
$$v'_{1}(0) = \tfrac{1}{4\pi^2} <v, (q^1)'(0)>_{L^2(g_{euc})},$$
as computed in  \cite{KuwertLorenz}. \\

Let $\;\eta:= \big(q^1\big)'(0) \; $ and $\;\eta^\circ = \eta_1 q^1 + \eta_2 q^2\;$ be its traceless part, then by Lemma 6 of \cite{KuwertLorenz} we have

\begin{equation}\label{etagleichung}
\begin{split}
(\text{div}_{euc} \eta^{\circ})_1 &= (\text{div}_{euc} u)_2\\
(\text{div}_{euc} \eta^{\circ})_2 &= (\text{div}_{euc} u)_1.
\end{split}
\end{equation}

For $\;u = u_1 q^1 + u_2 q^2\;$ we obtain,

$$ (\text{div}_{euc} u)_1 = \partial_2 u_1 - \partial_1 u_2, \quad (\text{div}_{euc} u)_2 = \partial_1 u_1 + \partial_2 u_2,$$

and therefore the Equations \eqref{etagleichung} become 

\begin{equation}\label{Ungleichung}
\begin{split}
 \partial_2 \eta_1 - \partial_1 \eta_2 &= \partial_1 u_1 + \partial_2 u_2\\
 \partial_1 \eta_1 + \partial_2 \eta_2 &= \partial_2 u_1 - \partial_1 u_2.
\end{split}
\end{equation} 

If we specialize to the relevant case $\;u = u_2 q^2\;$ and $\;v = v_2 q^2\;$ this yields 

$$\big(D^2\pi(g_{euc})(u, v)\big)_1 = \tfrac{1}{4\pi^2}< v_2 q^2 ,\eta^\circ>_{L^2(g_{euc})},$$ 

and we only need to concentrate on $\;\eta_2.\;$  Differentiating the Equations \eqref{Ungleichung} and subtracting these form each other gives (with $\;u_1 = 0$)

\begin{equation}\label{laplace}
\Delta \eta_2 = - 2 \partial_1\partial_2 u_2.
\end{equation}

In order to compute $\;\eta_2\; $ we restrict to normal variations $\;\varphi = \Phi \vec{n}\;$ for doubly periodic functions $\;\Phi\;$ in a Fourier space, i.e.,  $\;\Phi\;$ is a doubly periodic function on $\;\C\;$ with respect to the lattice $\;\sqrt{2}\pi \Z + \sqrt{2}\pi i \Z.\;$ The Fourier space $\;\mathcal F \big(T^2_1\big)\;$ of doubly periodic functions is the disjoint union of the constant functions and the $4$-dimensional spaces $\;\mathcal A_{kl}\big(T^2_1\big),\;$ $(k,l) \in \N\setminus \{(0,0)\}\;$ with basis

\begin{equation}
\begin{split}
&\sin\big(\sqrt{2} kx\big)\cos\big(\sqrt{2}ly\big), \quad \cos\big(\sqrt{2}kx\big)\sin\big(\sqrt{2}ly\big),\\ &\cos\big(\sqrt{2}kx\big)\cos\big(\sqrt{2}ly\big), \quad \sin\big(\sqrt{2}kx\big)\sin\big(\sqrt{2}ly\big).
\end{split}
\end{equation}

We restrict to the case where $\;\Phi = \Phi_{kl} \in \mathcal A_{kl},$ $(k,l) \in \N^2\setminus (0,0)\;$ in the following. Then for $\;u = v = \Phi_{kl} \vec{n}\;$ we obtain that 

$$\eta_2 = \tfrac{2}{k^2+ l^2}\partial_1\partial_2 \Phi_{kl}$$ 

solves equation \eqref{laplace}. The integration constant is hereby chosen such that $\;<\eta^{0}, q_1>_{L^2(g_{euc})} = 0.$
 
 Thus
 $$D^2\pi^1\big(G\big(f^1\big)\big)(u, v) = \tfrac{1}{2\pi^2 (k^2 + l^2)} \int_{T^2_1} \big(\partial^2_{12} \Phi_{kl}\big) \Phi_{kl}.$$
 
Put all calculations together we obtain

\begin{equation*}
\begin{split}
D^2\bigPi^1\big(f^1\big)\big(\varphi, \varphi\big)
&= -\tfrac{1}{\pi^2} \int_{T^2_1} \big(\partial^2_{12} \Phi_{kl}\big) \Phi_{kl}  \\
 &+ \tfrac{2}{\pi^2(k^2 + l^2)} \int_{T^2_1} \big(\partial^2_{12} \Phi_{kl}\big) \Phi_{kl}.
\end{split}
\end{equation*}

\begin{Rem}
The second variation for general normal variation $\;\varphi = \left(\sum_{k,l \in \N^2} a_{k,l} \Phi_{k,l}\right) \vec{n}\;$ is obtained by linearity. Terms obtain by pairing $\;\Phi_{k,l}\;$ and $\;\Phi_{m,n},\;$ where $\;(k,l) \neq (m,n)\;$ vanishes. To determine stability of $\;\mathcal W_\alpha\;$ we can thus restrict ourselves with out loss of generality to the case $\;\varphi = \Phi_{k,l} \vec{n}.$\\
\end{Rem}

Clearly, if for a normal variation $\;\varphi\;$ we have 

$$D^2\bigPi^1\big(f^1\big)\big(\varphi, \varphi\big) \leq 0,$$

 then by the stability of the Clifford torus 
 
$$D^2 \mathcal W_{\alpha}\big(f^1\big)\big(\varphi, \varphi\big) \geq 0$$  

for all $\;\alpha \geq 0.\;$ Moreover, for

\begin{equation*}
\begin{split}
\Phi_{kl} &= a \sin\big(\sqrt{2}kx\big)\cos\big(\sqrt{2}ly\big) + b \cos\big(\sqrt{2}kx\big)\sin\big(\sqrt{2}ly\big)\\ &+ c \cos\big(\sqrt{2}kx\big)\cos\big(\sqrt{2}ly\big) + d \sin\big(\sqrt{2}kx\big)\sin\big(\sqrt{2}ly\big)
\end{split}
\end{equation*}

with $\;k,l \in \N \setminus \{0\}\;$ and $\;a,\; b, \;c, \;d \in \R\;$ we have:

\begin{equation}\label{D^2Pi^1}
\begin{split}
D^2\bigPi^1\big(f^1\big)\big(\varphi, \varphi)  &=\tfrac{1}{\pi^2} \big(2kl - \tfrac{4kl}{k^2+l^2} \big)\tfrac{2ab - 2cd}{a^2 + b^2+ c^2+ d^2}<\varphi, \varphi>_{L^2}\\& \leq \tfrac{1}{\pi^2} \big(2kl - \tfrac{4kl}{k^2+l^2} \big)<\varphi, \varphi>_{L^2},
\end{split}
\end{equation}

with equality if and only if 
\begin{equation}\label{Pi2=}
a = b \quad \text{ and } \quad c = - d.
\end{equation}

\noindent
The second variation of the Willmore functional at the Clifford torus (Lemma 3 \cite{KuwertLorenz}) is given by:

\begin{equation}\label{D^2W}
\begin{split}
D^2 \mathcal W\big(f^1\big) (\varphi, \varphi) &= <\big(\tfrac{1}{2}\Delta^2 + 3\Delta + 4\big)\varphi,\varphi >_{L^2}\\
&= \big(2 (k^2 + l^2)^2 - 6(k^2 + l^2) + 4\big)<\varphi,\varphi>_{L^2}.
\end{split}
\end{equation}

Therefore we have

$$D^2 \mathcal W\big(f^1\big) (\varphi, \varphi) =  0,$$ if and only if
$\;k =\pm 1\;$ and $\;l = \pm 1,\;$
or $\;k = 0\;$ and $\;l = \pm1,\;$
 or $\;k = \pm1\;$ and $\;l = 0.$\\

\noindent
Let $\;c:= \tfrac{k}{l}\;$ and we assume without loss of generality that $\;c \geq 1,\;$ then the second variation formulas \eqref{D^2Pi^1} and \eqref{D^2W} simplifies to:

\begin{equation*}
\begin{split}
D^2 \mathcal W\big(f^1\big) (\varphi, \varphi) &= \big(2 (c^2 + 1)^2 l^4 - 6 (c^2 + 1) l^2 + 4\big)<\varphi,\varphi>_{L^2}\\
D^2\bigPi^1\big(f^1\big)(\varphi, \varphi)  &\leq  \tfrac{1}{\pi^2} \big(2c l^2 -4 \tfrac{c}{c^2 +1}\big)<\varphi,\varphi>_{L^2}.
\end{split}
\end{equation*}

Hence we obtain for $\; \tilde\alpha = \tfrac{1}{4\pi^2} \alpha$

$$D^2 \mathcal W_{\alpha} (f^1)(\varphi, \varphi) \geq \Big(2 (c^2 + 1)^2 l^4 - \big(6(c^2 + 1) +  8\tilde \alpha c\big)l^2 + 4 + 16\tilde \alpha \tfrac{c}{c^2 +1} \Big) <\varphi,\varphi>_{L^2}.$$

with equality is and only if $\;\Phi\;$ is satisfies \eqref{Pi2=}.
We still want to determine the range of $\;\alpha\;$ for which $\;\mathcal W_{\alpha}\;$ is stable.  At $\;\alpha = \alpha^b\;$ the second variation of $\;\mathcal W_{\alpha}\;$ have zero directions in the normal part which are not  M\"obius variations. Thus we need to determine the roots the polynomial 

$$g_{\tilde \alpha, c}(l): = \Big(2 (c^2 + 1)^2 l^4 - \big(6(c^2 + 1) +  8\tilde \alpha c\big)l^2 + 4 + 16\tilde \alpha \tfrac{c}{c^2 +1} \Big)$$

The polynomial $\;g_{\tilde \alpha, c}\;$ is even, its leading coefficient is positive and its roots satisfy:

\begin{equation}\label{rootsl}
l^2 = \tfrac{2}{c^2 +1}, \text{ or } \quad l^2= \tfrac{1}{c^2 +1} + 4 \tilde \alpha\tfrac{c}{(c^2 +1)^2}.
\end{equation}

The values of $\;l \in \N\;$ for which $\;g_{\tilde \alpha, c}\;$ is negative lies exactly between the positive roots of $\;g_{\tilde \alpha, c}.\;$ So we want to determine $\;\tilde \alpha\;$ such that this region of negativity  for $\;g_{\tilde \alpha, c},\;$ i.e., the intervall between the two positive solutions $\;l_1(\tilde \alpha, c)\;$ and $\;l_2(\tilde \alpha, c)\;$ of \eqref{rootsl} contain no positive integer for all $\;c  = \tfrac{k}{l}\;$ (other than those combinations leading to a M\"obius variation).
\noindent 
We consider two different cases:

$$c= 1 \text{ and } c > 1.$$

For $\;c = 1\;$ the four roots of $\;g_{\tilde \alpha,1}\;$ are determined by:  

$$l^2 =1, \quad l^2 = \tfrac{1}{2} + \tilde \alpha.$$

Since the case of $\;l^2 = 1,\;$ i.e., $\;l=k = \pm1\;$ corresponds to M\"obius variations, we can rule out the existence of negative values of $\;g_{\tilde \alpha, 1}\;$ if and only if the second root satisfies 

$$|l| \leq 2, \quad \text{ or equivalently, } \quad  l^2 = \tfrac{1}{2} + \tilde \alpha \leq 4.$$

 From which we obtain $\;\tilde \alpha \leq \frac{7}{2}.$\\

For $\;c > 1,\;$ the first equation $\;l^2 = \tfrac{2}{c^2 +1}< 1\;$ is never satisfied for an integer $\;l.\;$ Thus we only need to consider the equation
$$l^2 = \tfrac{1}{c^2 +1} + 4 \tilde \alpha\tfrac{c}{(c^2 +1)^2}.$$

To rule out negative directions for $\;D^2\mathcal W_{4 \pi^2 \tilde \alpha}\;$ it is necessary and sufficient to have 

$$l^2 = \tfrac{1}{c^2 +1} + 4 \tilde \alpha\tfrac{c}{(c^2 +1)^2} \leq 1 $$

for appropriate $\;c= \tfrac{k}{l}.\;$ For $\;l^2 = 1\;$ we obtain that $\;c = \tfrac{k}{l} \in \N_{>1}\;$ and $\;\tilde \alpha\;$ satisfies:

$$\tilde \alpha  =  \tfrac{1}{4} (c^3+ c).$$

The right hand side is monotonic in $\;c\;$ and therefore the minimum for $\;c \in \N_{>1}\;$ is attained at $\;c=2\;$ which is equivalent to $\;\tilde \alpha  =\frac{5}{2}.\;$ Since $\;\frac{5}{2} < \tfrac{7}{2}\;$ which was the maximum $\;\tilde \alpha\;$ in the $\;c=1\;$ case, we get that  $\;\delta^2\mathcal W_{10 \pi^2}\big(f^1\big) \geq 0.\;$ Further, at $\;\tilde \alpha = \frac{5}{2}\;$ the (non-M\"obius) normal variations in the kernel of $\;\delta^2\mathcal W_{10 \pi^2}\big(f^1\big)\;$ are given by

\begin{equation}
\begin{split}
\Phi_1 &= \sin(2 \sqrt{2} y) \cos(\sqrt{2} x) +  \cos(2\sqrt{2} y) \sin(\sqrt{2} x) = \sin\big(\sqrt{2}(x+ 2y)\big)\\
\tilde  \Phi_1  &= \sin(2 \sqrt{2} y) \sin(\sqrt{2} x) - \cos(2\sqrt{2} y) \cos(\sqrt{2} x) = \cos\big(\sqrt{2}(x+2y)\big)
\end{split}
\end{equation}

and by symmetry of $\;k\;$ and $\;l\;$ (we have assumed $\;c\geq1$): 

\begin{equation}
\begin{split}
\Phi_2 &= \sin(2 \sqrt{2} x) \cos(\sqrt{2} y) +  \cos(2\sqrt{2} x) \sin(\sqrt{2} y) = \sin\big(\sqrt{2}(2x + y)\big)\\
\tilde  \Phi_2  &= \sin(2 \sqrt{2} x) \sin(\sqrt{2} y) - \cos(2\sqrt{2} x) \cos(\sqrt{2} y)= \cos\big(\sqrt{2}(2x+y)\big),
\end{split}
\end{equation}

where $\;\tilde \Phi_i (x,y) = \Phi_i (x, y+ \tfrac{\pi}{2}),\;$ i.e., $\;\Phi_i$ and $\tilde \Phi_i\;$ differ only by a translation. We have shown the following Lemma.\\

\begin{Lem}\label{kernvonalpha}
At $\;b = 1\;$ we have that 

$$\alpha^1 = \text{ max }\big\{\;\alpha >0 \  | \ \delta^2\mathcal W_{\alpha}\big(f^1\big) \geq 0\; \big\}$$

is computed to be $\;10 \pi^2.\;$\\
\end{Lem}

The problem at $\;b= 1\;$ is that the kernel dimension of $\;\delta^2 \mathcal W_{\alpha^1}\big(f^1\big)\;$ is too high. Even using the invariance of the equation it is not possible to reduce it to $1$, which is needed for the bifurcation theory from simple eigenvalues. The main reason is that linear combinations of the two $\;\Phi_i\;$ cannot be reduced to a translation and scaling of $\;\Phi_1\;$ only. This situation is different for $\;b\neq 1,\;$ see Proposition \ref{1d}, because for homogenous tori \eqref{parametrizationofhomogenous} the immersion is not symmetric w.r.t. parameter directions $\;x\;$ and $\;y\;$. For $\;b \neq 1\;$ we have that $\;\beta^b \neq 0\;$ and thus the second variation of $\;\bigPi^2\;$ enters the calculation of  
$$\alpha^b= \text{ max }\big\{\;\alpha>0 \;| \; \delta^2\mathcal W_{\alpha, \beta^b} \geq0\;\big\}.$$

Moreover, $\;A_{k,l}\big( T^2_1\big)\;$ is canonically isomorphic to $\;A_{k,l}\big( T^2_b\big)\;$ via

\begin{equation}
\begin{split}
\sin\big(k \sqrt{2}x\big)\cos\big(l \sqrt{2}y\big) &\longmapsto \sin\big(\tfrac{k x}{r}\big)\cos\big(\tfrac{l y}{s}\big),\\
\sin\big(k \sqrt{2}x\big)\sin\big(l \sqrt{2}y\big) &\longmapsto \sin\big(\tfrac{k x}{r}\big)\sin\big(\tfrac{l y}{s}\big),\\
\cos\big(k \sqrt{2}x\big)\sin\big(l \sqrt{2}y\big) &\longmapsto \cos\big(\tfrac{k x}{r}\big)\sin\big(\tfrac{l y}{s}\big),\\
\cos\big(k \sqrt{2}x\big)\cos\big(l \sqrt{2}y\big) &\longmapsto \cos\big(\tfrac{k x}{r}\big)\cos\big(\tfrac{l y}{s}\big).
\end{split}
\end{equation}

To emphasis this isomorphism, we denote in the following normal variations at the Clifford torus by $\;\varphi^1 = \Phi^1 \vec{n}^1,\;$ with $\;\Phi^1\;$ a well defined function on $\;T^2_1,\;$  and the corresponding normal variations at homogenous tori $f^b$ under the above isomorphism by $\;\varphi^b = \Phi^b \vec{n}^b.\;$
Since $\;\delta^2 \mathcal W \geq0\;$ and $\;\delta^2\bigPi^1 \geq0\;$
we obtain the following Lemma using Lemma 4 and 7 of \cite{KuwertLorenz}.\\

\begin{Lem}\label{spectralproperty}
With the notations as above let $\;\alpha_0 \in \R_+\;$ be fixed and $\;\Phi^1 \in A_{k,l}(T^2_1)\;$ such that 

$$\delta^2\mathcal W_{\alpha_0}\big(f^1\big)\big( \varphi^1,  \varphi^1\big) >0.$$

Then for $\;b\sim1\;$ close enough we also have

$$\delta^2\mathcal W_{\alpha, \beta^b}\big(f^b\big)\big(  \varphi^b, \varphi^b\big)>0,$$

for all $\;\alpha \leq \alpha_0.\;$\\
\end{Lem}

Kuwert and Lorenz  \cite{KuwertLorenz} computed the second derivative of $\;\bigPi^2\;$ for $\;\varphi^b = \Phi^b_{k,l} \vec{n}^b \;$ to be 

\begin{equation}
\begin{split}
D^2 \bigPi^2\big(f^b\big)\big( \varphi^b,  \varphi^b\big) &=\tfrac{1}{4\pi^2 r^2} \int_{T^2_b}<\del^2_{11} \Phi^b- \del^2_{22} \varphi^b,  \varphi^b>dA\\
 &+ \tfrac{r^2-s^2}{4\pi^2 r^4s^2}\int_{T^2_b} | \varphi^b|^2dA \\
&- \tfrac{2(r^2-s^2) + c_r(k,l)}{4\pi^2 r^4s^2} \int_{T^2_b} | \varphi^b|^2dA,
\end{split}
\end{equation}

where $\; c_r(k,l) := \tfrac{k^2s^2- l^2r^2}{k^2s^2+l^2r^2} \;$ and $\; \varphi^b \in A_{k,l}\big(T^2_b\big) \vec{n}^b.$

For $\;b \sim 1,\;$ i.e., $\;r \sim \tfrac{1}{\sqrt{2}}\;$ this yields 

$$D^2\bigPi^2\big(f^b\big)\big( \varphi^b_1,  \varphi^b_1\big) >D^2 \bigPi^2\big(f^b\big)\big( \varphi^b_2,  \varphi^b_2\big),$$

for $\; \varphi_i^b\;$ are the images of $ \;\varphi_i^1 \in $ Ker $\delta^2\mathcal W_{\alpha^1}\big(f^1\big)\;$ under the canonical isomorphism and since for  $\;b>1,\;$ i.e.,  $\;r<s\;$ and $\;\beta^b >0\;$ we obtain

$$\delta^2 \mathcal W_{\alpha^1, \beta^b}\big(f^b\big)\big( \varphi^b_1,  \varphi^b_1\big) < \delta^2 \mathcal W_{\alpha^1, \beta^b}\big(f^b\big)\big( \varphi^b_2,  \varphi^b_2\big) < 0.$$

Thus $\;\alpha^b < \alpha^1\;$ and we obtain that for $\;b \sim 1\;$ and $\;b \neq 1\;$ the kernel of $\;\delta^2\mathcal W_{\alpha^b, \beta^b}\big(f^b\big)\;$ is $2$-dimensional and consists of either $\; \varphi^b_1\;$ and $\;\tilde  \varphi^b_1\;$ for $\;b>1\;$  or $ \;\varphi^b_2\;$ and $\; \varphi^b_2\;$ for $\;b<1.\;$ Both choices of $\;b\;$ lead to M\"obius invariant surfaces. We summarize the results in the following Lemma:\\

\begin{Lem}
For $\;b \sim1,\;$ $\;b>1\;$ we have that $\;\alpha^b\;$ is uniquely determined by the kernel of 
$\;\delta^2 \mathcal W_{\alpha^b, \beta^b}\big(f^b\big)\;$ which is $2$ dimensional and spanned  (up to invariance) by the normal variations 

$$ \varphi^b_1=  \sin\big(\sqrt{2} (\tfrac{x}{r} + \tfrac{2y}{s})\big) \vec{n}^b\quad \text{and} \quad \tilde \varphi^b_1 =  \cos\big(\sqrt{2} (\tfrac{x}{r} + \tfrac{2y}{s})\big)\vec{n}^b.$$

\end{Lem}
Now, for $\;b\sim 1\;$ consider the reparametrization of the homogenous torus as a $(2,-1)$-equivariant surface

\begin{equation*}
\begin{split}
\tilde f^b: \C /\big(2\pi \Z + 2\pi \tfrac{2r^2 + isr}{4r^2 + s^2} \Z\big) & \longrightarrow S^3\subset \C^2, \\
 (\tilde x, \tilde y)&\longmapsto \begin{pmatrix}r e^{i2\tilde y +  \tfrac{i s \tilde x}{r}}, s e^{-i\tilde y+  \tfrac{   i  2r \tilde x}{s}}\end{pmatrix}.
 \end{split}
 \end{equation*}

Using these new coordinates the kernel of $\;\delta^2 \mathcal W_{\alpha^b, \beta^b}\big(f^b\big)\;$ for $\;b = \tfrac{s}{r}>1\;$ is given by

$$\Phi^b= \sin \big((\tfrac{s}{r}+ 4 \tfrac{r}{s})\tilde x\big), \quad \tilde \Phi^b = \cos\big((\tfrac{s}{r}+ 4 \tfrac{r}{s})\tilde x\big).$$

Thus infinitesimally the $\;\tilde y$-direction of the surface is not affected by a deformation with normal variation $\;\Phi^b\vec{n}^b,\;$ i.e., the $\;(2, -1)$-equivariance is infinitesimally preserved. Since the space of $\;(2,-1)$-equivariant surfaces and $\;(1,2)$-equivariant surfaces are isomorphic and differ only by the orientation of the surface and an isometry of $\;S^3\;$, we will consider $\;(1,2)$-equivariant surfaces for convenience. Moreover, it is important to note that for all real numbers $\;c_1,\;$ $c_2\;$ there exist $\;d_1, d_2 \in \R\;$ such that

\begin{equation}\label{additionstheoreme}
\begin{split}
c_1 \Phi_1^b + c_2 \tilde \Phi_1^b &= c_1 \sin\big((\tfrac{s}{r}+ 4 \tfrac{r}{s})\tilde x\big) + c_2 \cos\big((\tfrac{s}{r}+ 4 \tfrac{r}{s})\tilde x\big) \\
&= d_1 \sin\big((\tfrac{s}{r}+ 4 \tfrac{r}{s})\tilde x+ d_2\big) =  d_1\Phi_1^b\big((\tfrac{s}{r}+ 4 \tfrac{r}{s})\tilde x + d_2\big).
\end{split}
\end{equation}

Since homogenous tori $\;\tilde f^b\;$ satisfy $\;\tilde f^b(\tilde x + d_2,\tilde y) = M \tilde f^b(\tilde x, \tilde y),\;$ where $\;M\;$ is a isometry of $\;S^3,\;$ we obtain the following proposition reducing the kernel dimension of $\;\delta^2\mathcal W_{\alpha^b, \beta^b}\big(f^b\big)\;$ to $1$ (up to invariance). \\

\begin{Pro}\label{1d}

For a family of $f^b_{(s,t)} = \exp_{f^b}\big(t \varphi^b_1 + s \tilde \varphi^b_1\big)$ be a family of immersions from $T^2_b \longrightarrow S^3.$ Then there exist M\"obius transformations $\;M(s,t),\;$ reparametrizations $\;\sigma(s,t),\;$ and a function $\;c(s,t)\;$ such that

$$M(s,t) \circ f^b_{(s,t)} \circ \sigma(s,t)  = \exp_{f^b}\big(d_1(s,t) \varphi^b\big)$$

\end{Pro}

\begin{proof}

Let $\;\varphi^b = \big(s \Phi^b_1 + t \tilde \Phi^b_1\big) \vec{n}^b(\tilde x, \tilde y).\;$ Then
by Equation \eqref{additionstheoreme} we obtain real functions  $\;d_1(s,t)\;$ and $\;d_2(s,t)\;$ satisfying 

$$\varphi^b =\big(d_1(s,t) \Phi^b_1\big(\tilde x+ d_2(s,t)\big)\big) \vec{n}^b(\tilde x, \tilde y).$$

By definition of the homogenous tori there exist a isometry $\;M(s,t)\;$ of $\;S^3\;$ such that $\;M(s,t)\circ f^b = f^b(\tilde x+d_2(s,t), \tilde y).\;$ Thus $\;M\;$ induces a map, which we again denote by $M,$ on the normal vector $\;\vec{n}^b\;$ given by 

 $$M\circ \big(\vec{n}^b(x,y)\big) = \vec{n}^b(\tilde x + d_2(s,t), \tilde y).$$
 
 Therefore, $\;M \circ \varphi^b = \big(d_1 \Phi^b_1(\tilde x+ d_2)\big) \vec{n}^b(\tilde x+d_2, \tilde y)\;$ and with 
 
 $$\sigma(s,t) :  T^2_b \longrightarrow T^2_b , \quad (\tilde x, \tilde y) \longmapsto (\tilde x - d_2, \tilde y )$$

\noindent 
we hence obtain the desired property.
\end{proof}

\section{A classification of constrained Willmore tori}
Before classifying all solutions to the Euler-Lagrange equation \eqref{ELana} with control on the Lagrange multiplier, we first show a technical lemma that allow us to use Bifurcation Theory. 
\begin{Lem}\label{n-thderivative}
For $\;b \sim 1\;$ we obtain with the notations introduced in Section \ref{strategy} and \ref{stability}

 $$\delta^3 \mathcal W_{\alpha^b, \beta^b}\big(f^{b}\big)\big(\varphi^b, \varphi^b, \varphi^b\big) = 0.$$
 
Moreover, the fourth variation of the Willmore functional satisfies

$$ \delta^4 \mathcal W_{\alpha^b, \beta^b}\big(f^{b}\big)\big(\varphi^b,\cdots, \varphi^b\big) + \delta^3 \mathcal W_{\alpha^b, \beta^b}\big(f^{b}\big)\big(\del_{\tilde a}\varphi(a)|_{\tilde a=0}, \varphi^b, \varphi^b\big) \neq 0.
\footnote{Recall that $\;\varphi(a)= \big(\del_{\tilde a} \tilde f^{(a,b)}\big)^\perp$ for as family $\;\tilde f^{(a,b)}\;$ deforming $f^b.$}$$
\end{Lem}

 \begin{proof}
For fixed $\;b \sim 1\;$ and candidate surfaces $\;f_{(a,b)}\;$ constructed in Section 4 let $\;\varphi(a) := \big(\del_{\sqrt{a}} f_{(a,b)}\big) ^\perp.\;$ This implies 

\begin{equation}\label{pieq}
\delta \bigPi^2\big(f_{(a,b)}\big)\varphi(a) = 0 \quad \text{ and } \quad \delta \bigPi^1\big(f_{(a,b)}\big) \big(\varphi(a)\big) = 2 \sqrt{a}.
\end{equation} 

Further, recall that $\;\varphi^b = \lim_{a \rightarrow 0}\varphi(a)\;$ and  $\;\alpha_{(a,b)}\;$ and $\;\beta_{(a,b)}\;$ are the Lagrange multipliers of the candidate surfaces with

\begin{equation}\label{alphabetaeq}
\alpha^b= \lim_{a \rightarrow 0} \alpha_{(a,b)} \quad \text{ and } \quad \beta^b= \lim_{a \rightarrow 0} \beta_{(a,b)}.
\end{equation} 

The surfaces $\;f_{(a,b)}\;$ all satisfy the Euler Lagrange equation  \eqref{ELana}. Therefore, testing  \eqref{ELana} with $\;\varphi(a)\;$ gives

\begin{equation}\label{firstvariation}
\delta \mathcal W\big(f_{(a,b)}\big)\big(\varphi(a)\big) = \alpha_{(a,b)} \delta\bigPi^1\big(f_{(a,b)}\big) \big(\varphi(a)\big) + \beta_{(a,b)} \delta\bigPi^2\big(f_{(a,b)}\big)\big(\varphi(a)\big).
\end{equation}

Differentiate the above equation with respect to $\;\sqrt{a}\;$ together with the Euler-Lagrange equation  yields

\begin{equation}
\begin{split}
\delta^2 \mathcal W\big(f_{(a,b)}\big)\big(\varphi, \varphi\big) &= \alpha_{(a,b)} \delta^2 \bigPi^1\big(f_{(a,b)}\big)\big(\varphi, \varphi\big)  + \beta_{(a,b)} \delta^2 \bigPi^2\big(f_{(a,b)}\big)\big(\varphi, \varphi\big)\\& +\del_{\sqrt{a}}  \alpha_{(a,b)} \delta \bigPi^1\big(f_{(a,b)}\big)\big(\varphi\big) + \del_{\sqrt{a}} \beta_{(a,b)} \delta \bigPi^2\big(f_{(a,b)}\big)\big(\varphi\big). 
\end{split}
\end{equation}

Differentiating once again and evaluating at $\;a = 0\;$ combined with \eqref{pieq} and \eqref{alphabetaeq} results in the following equation for the third derivative:

\begin{equation}
\begin{split}
\delta^3 \mathcal W \big(f^b\big)\big(\varphi^b, \varphi^b, \varphi^b\big) &=  \alpha^b\delta^3 \bigPi^1\big(f^b\big)\big(\varphi^b, \varphi^b, \varphi^b\big) + \beta^b\delta^3 \bigPi^2\big(f^b\big)\big(\varphi^b, \varphi^b, \varphi^b\big)\\ &+ 2 \lim_{a \rightarrow 0}\del_{\sqrt{a}} \alpha_{(a,b)} \delta^2 \bigPi^1\big(f^b\big)\big(\varphi^b, \varphi^b\big) + 2\lim_{a \rightarrow 0}\del_{\sqrt{a}} \beta_{(a,b)} \delta^2 \bigPi^2\big(f^b\big) \big(\varphi^b, \varphi^b\big).
\end{split}
\end{equation}

By \eqref{pieq} we have $\;\delta^2 \bigPi^1\big(f^b\big)\big(\varphi^b, \varphi^b\big) = 2\;$ and by \cite[Lemma 2.2] {HelNdi2} the candidates satisfies

 $$\lim_{a\rightarrow 0}\del_{\sqrt{a}}\alpha_{(a,b)}=0 \quad \text{ and } \quad \lim_{a\rightarrow 0}\del_{\sqrt{a}}\beta_{(a,b)}=0.$$

Therefore, we obtain 

$$\delta^3 \mathcal W_{\alpha^b, \beta^b} \big(f^b\big)\big(\varphi^b, \varphi^b, \varphi^b\big) =0.$$

Differentiating the equation \eqref{firstvariation} three times and taking the limit for $\;\tilde a: = \sqrt{ a} \longrightarrow 0\;$ gives the following formula:

\begin{equation}
\begin{split}
\delta^4 \mathcal W_{\alpha^b, \beta^b}\big(f^b\big)\big(\varphi^b,\cdots, \varphi^b\big) + \delta^3 \mathcal W_{\alpha^b, \beta^b}\big(f^b\big)\big(\del_{\tilde a}\varphi|_{\tilde a= 0}, \varphi^b, \varphi^b\big) = \\\lim_{\tilde a \rightarrow 0}\del^2_{\tilde a \tilde{a}} \alpha_{(a,b)} \delta^2\bigPi^1\big(f^b\big)\big(\varphi^b, \varphi^b\big)  + \lim_{\tilde a \rightarrow 0}\del^2_{\tilde a \tilde{a}} \beta_{(a,b)} \delta^2\bigPi^2\big(f^b\big)\big(\varphi^b, \varphi^b\big).
\end{split}
\end{equation}

We have computed for the candidates that 

$$\lim_{\tilde a\rightarrow 0}\del^2_{\tilde a \tilde{a}}\beta_{(a,b)} = \del_b \alpha^b \leq 0 \quad \text{and} \quad  \lim_{\tilde a\rightarrow}\del^2_{\tilde a \tilde{a}}|_{\tilde a=0}  \alpha_{(a,b)} = 2 \del_a|_{a=0}  \alpha_{(a,b)}<0.$$

 Together with $$\;\delta^2 \bigPi^1\big(f^b\big) (\varphi^b, \varphi^b) >0\quad \text{ and }\del_b \alpha^b|_{b=1} = 0\;$$ we conclude that the second formula of the Proposition holds for $\;b \sim 1$.

\end{proof}

Now we can turn to the main theorem of the section.

\begin{The} \label{cwm}
For $\;b\sim 1\;$ and $\;b \neq 1\;$ fixed there exist a $\;a^b >0\;$ such that there exists a unique branch  of solution (up to invariance) to the Euler-Lagrange equation

 \begin{equation}\label{locpro}
\begin{split}
\delta \mathcal W _{\alpha, \beta} (f) = 0, \quad &\text{with} \quad \alpha \sim \alpha^b, \;\alpha \leq \alpha^b, \;\beta \sim \beta^b, \;f \sim f^b\; \text{ smoothly}\\
&\text{and} \quad  \bigPi^1(f) = a, \;\bigPi^2(f) = b\; \text{ with } \;b \sim 1 \;\text{ fixed and } \;a \in [0, a^b).
\end{split}
\end{equation}

In particular, for $\;\alpha = \alpha^b\;$ and $\;\bigPi^2(f) = b\;$ the only solution of \eqref{locpro} is the homogenous torus.
\end{The}

\begin{proof}
We prove the above theorem using  Bifurcation Theory from Non Linear Analysis, more precisely bifurcation from simple eigenvalues, see \cite{AP}.

We subdivide the proof into the following four steps:
 \begin{enumerate}
\item the splitting of the Euler-Lagrange equation \eqref{locpro} into an auxiliary and a bifurcation part,\\
\item classification of all solutions to the auxiliary equation,\\
\item classification of all solutions to the bifurcation equation,\\
\item identification of the  Teichm\"uller class of the previously obtained solutions.
\end{enumerate}

We first fix some further notations: we will work on the following Sobolev space given by

$$W^{4,2}\big(T^2_b, S^3\big) := \big\{\;V : T_b^2 \longrightarrow S^3 \subset \R^4 \; \mathlarger{|} \;\text{ each } V^i \in W^{4,2}\big(T^2_b, \R\big)\;\big \},$$

where $\;W^{4,2}\big(T^2_b, \R\big)\;$ is the usual Sobolev space, namely

\begin{equation*}
\begin{split}
W^{4,2}\big(T^2_b, \R\big) := \big\{\; V: T^2_b \longrightarrow \R\;\mathlarger{|}\; V \text{ and its derivatives up to order 4 are all } \\
L^2 \text{ integrable with respect to } g_b = (f^b)^*\big(g_{S^3}\big)\;\big\}.
\end{split}
\end{equation*}

Since tangential variations only lead to a reparametrization of the surface preserving $\;\mathcal W\;$ and $\;\bigPi\;$ we can restrict ourselves to the space

$$W^{4,2, \perp}\big(T^2_b, S^3\big) := \big\{ \; V\in W^{4,2}\big(T^2_b, S^3\big)\; \mathlarger{|}\; V \perp df^b \; \text{ on }\; T^2_b\; \big\}.$$

Further, for an appropriate neighborhood $\;U(0)\;$ of 

$$0 \;\in\; <\varphi^b, \tilde \varphi^b>^{\perp, W^{4,2,\perp}}:= \big\{ \text{ orthogonal complement of } \;<\varphi^b, \tilde \varphi^b> \;\text{ in } W^{4,2,\perp} \text{ with } W^{4,2}\text{-topology }\big\}$$

we consider the map

$$\Phi : U(0) \times \R^2 \times \R \times \R\longrightarrow L^{2, \perp}\big(T_b^2, S^3\big) := L^2\big(T_b^2, S^3\big) \cap W^{4,2,\perp}\big(T_b^2, S^3\big),$$

given by 

\begin{equation*}
\begin{split}
\Phi(V, \alpha, \beta, t, s) &= \delta \mathcal W_{\alpha, \beta}\big(\exp_{ f^b} \big( V + t\varphi^b+ s \tilde \varphi^b\big)\big) 
\\ &=\delta \mathcal W\big( \exp_{ f^b} \big( V + t\varphi^b+ s \tilde \varphi^b\big)\big)  \\&- \alpha \delta \bigPi^1\big( \exp_{ f^b} \big( V + t\varphi^b+ s \tilde \varphi^b\big)\big) \\&- \beta \delta \bigPi^2\big( \exp_{ f^b} \big( V + t\varphi^b+ s \tilde \varphi^b\big)\big),
\end{split}
\end{equation*}

where $\;L^2\big(T^2_b, S^3\big) := \big\{ f: T^2_b \longrightarrow S^3 \;|\; f^i \in L^2\big(T^2_b, \R\big)\big\}\;$ and $\;L^2\big(T^2_b, \R\big)\;$ is the usual $\;L^2$-Lebesgue space. 
By \cite{NdiayeSchätzle1} the map $\;\Phi\;$ is smooth in $\;W^{4,2}$-topology and the solutions of \eqref{locpro} are exactly the zero locus of $\;\Phi.$

\subsection{Step (1)} $\ $\\We first observe that

$$\Phi\big(0, \alpha^b, \beta^b, 0, 0\big) = \delta \mathcal W_{\alpha^b, \beta^b}\big(f^b\big) = 0.$$ 

Moreover, $\;\alpha^b\;$ is chosen such that the homogenous tori are stable with respect to the functional  $\;\mathcal W_{\alpha^b. \beta^b}\;$, see Section \ref{stability}, and we have 

\begin{equation}\label{positivity1}
\partial_V \Phi \big(0, \alpha^b, \beta^b, 0, 0\big) \cdot Z = \delta^2 \mathcal W_{\alpha^b, \beta^b}\big(f^b\big)\big(Z, . \big).
\end{equation}

Moreover, $\;\partial_V \Phi \big(0, \alpha^b, \beta^b, 0, 0\big)\;$ is a Fredholm operator of index $\;0\;$ by \cite{NdiayeSchätzle1}. The stability computations in Section \ref{stability} further shows 

\begin{equation}\label{positivity2}
\begin{split}
&\delta^2 \mathcal W_{\alpha^b, \beta^b}\big(f^b\big) \big(Z,Z\big) \geq 0,  \\
\text{ and moreover } \quad \quad &\delta^2 \mathcal W_{\alpha^b, \beta^b}\big(f^b\big) \big(Z,Z\big) = 0 \\
\Leftrightarrow \; \quad \quad&Z \in \;<\varphi^b, \tilde \varphi^b> \oplus \text{ Moeb}_{f^b}T^2_b \oplus T_{f^b} T^2_b.
\end{split}
\end{equation}

Thus we obtain with the same arguments in the proof of equation (3.20) of \cite{NdiayeSchätzle1} that

\begin{equation}\label{I} 
\text{Ker }\big(\partial_V \Phi (0, \alpha^b, \beta^b, 0, 0)\big) =\; <\varphi^b, \tilde \varphi^b>  \oplus \text{ Moeb}_{f^b}T^2_b \oplus T_{f^b}T^2_b.
\end{equation}

On the other hand, using the symmetry of $\;\delta^2 \mathcal W_{\alpha^b, \beta^b}\big(f^b\big)\;$ and by the arguments of the proof of formula (3.21) of \cite{NdiayeSchätzle1} we get

\begin{equation}\label{II}
<\varphi^b, \tilde \varphi^b>\oplus \text{ Moeb}_{f^b}T^2_b \perp \text{ Im}\big(\partial_V \Phi (0, \alpha^b, \beta^b, 0)\big)  \text{ in } L^{2, \perp}\big(T_b^2, S^3\big).
\end{equation}

However, since $\;\partial_V \Phi \big(0, \alpha^b, \beta^b, 0, 0\big)\;$ is Fredholm with index $\;0\;$ we obtain by \eqref{I}

\begin{equation}\label{III}
\begin{split}
\dim\left( \bigslant{L^{2,\perp}\big(T_b^2, S^3\big)}{\text{ Im}\Big(\partial_V \Phi \big(0, \alpha^b, \beta^b, 0, 0\big) \Big)} \right) = \dim \big(\text{ Moeb}_{f^b} T_b^2 \; \oplus<\varphi^b, \tilde \varphi^b>\big) \\
= \dim \left ( \bigslant{L^{2,\perp}\big(T_b^2, S^3\big)}{\big(\text{ Moeb}_{f^b} T_b^2  \;\oplus <\varphi^b, \tilde \varphi^b>\big)^{\perp, L^{2,\perp}\big(T_b^2, \;S^3\big)}}\right).
\end{split}
\end{equation}

Together with Property \eqref{II} this yields 

$$\text{ Im}\Big(\partial_V \Phi \big(0, \alpha^b, \beta^b, 0, 0\big) \Big) = \big(\text{Moeb}_{f^b} T_b^2 \; \oplus <\varphi^b, \tilde \varphi^b>\big)^{\perp, L^{2,\perp}\big(T_b^2, \;S^3\big)}.$$

Let

$$Y := \big(\text{Moeb}_{f^b} T_b^2  \;\oplus <\varphi^b, \tilde \varphi^b>\big)^{\perp, L^{2, \perp}\big(T_b^2, \; S^3\big)}.$$

Since $\;\text{Moeb}_{f^b} T_b^2 \;\oplus\; <\varphi^b, \tilde \varphi^b>\;$ is finite dimensional we obtain 

$$L^2\big(T_b^2, S^3\big)^\perp = Y \oplus \text{ Moeb}_{f^b} T_b^2  \;\oplus \;<\varphi^b, \tilde \varphi^b>,$$

and thus

$$L^2\big(T_b^2, S^3\big) =  Y \oplus \text{ Moeb}_{f^b} T_b^2 \; \oplus <\varphi^b, \tilde \varphi^b> \;\oplus \; T_{f^b}T^2_b.$$

The above splitting still holds (though not as orthogonal decomposition) for 

$$V \in U(0) \subset  W^{4,2, \perp}\big(T^2_b, S^3\big) \subset \mathcal C^1\big(T_b^2, S^3\big)$$

and $\;t,\; s\;$ small (see proposition B.3 of \cite{NdiayeSchätzle1}), i.e, 

\begin{equation}\label{IV}
L^2\big(T_b^2, S^3\big) =  Y \oplus \text{ Moeb}_{\exp_{ f^b} \big( V + t\varphi^b+ s \tilde \varphi^b\big)} T_b^2 \oplus <\varphi^b, \tilde \varphi^b> \oplus T_{\exp_{ f^b} \big( V + t\varphi^b+ s \tilde \varphi^b\big)}T^2_b.
\end{equation}

On the other hand, since $\;\text{ Moeb}_{\exp_{ f^b} \big( V + t\varphi^b+ s \tilde \varphi^b\big)} T_b^2 \oplus<\varphi^b, \tilde \varphi^b>\;$ is finite dimensional we obtain for 

$$X:= \big(\text{ Moeb}_{f^b}T^2_b \oplus <\varphi^b, \tilde \varphi^b>\big)^{\perp, W^{4,2, \perp}\big(T^2_b, \; S^3\big)} \subset  W^{4,2, \perp}\big(T^2_b, \; S^3\big)$$

an analogous splitting for $\;W^{4,2},\;$ i.e., 

$$X \oplus \text{ Moeb}_{f^b} T_b^2 \;\oplus <\varphi^b, \tilde \varphi^b>\;= W^{4,2, \perp}(T^2_b, S^3).$$

To continue we define the following projection maps:

\begin{equation}
\begin{split}
&\bigPi_Y :  L^2\big(T^2_b, S^3\big)^\perp \longrightarrow Y,\\
&\bigPi_{\text{Moeb}_{f^b}T^2_b\oplus <\varphi^b, \tilde \varphi^b>} :  W^{4,2, \perp}\big(T^2_b, S^3\big) \longrightarrow \text{ Moeb}_{f^b}T^2_b \; \oplus <\varphi^b, \tilde \varphi^b>,\\
\text{ and } \quad &\bigPi_X : W^{4,2, \perp}\big(T^2_b, S^3\big) \longrightarrow X.
\end{split}
\end{equation}

This splitting \eqref{IV} ensures that we can decompose the equation $\;\Phi = 0\;$ close to $\;\big(0, \alpha^b, \beta^b, 0, 0\big)\;$ into two equations which we solve successively in the following:

\begin{equation}
\begin{cases}
&\bigPi_Y \Phi = 0\\
&\bigPi_{\text{Moeb}_{f^b}T^2_b\oplus <\varphi^b, \tilde \varphi^b>} \Phi = 0.
\end{cases}
\end{equation}

In the language of Bifurcation Theory the first equation is called the Auxiliary Equation and the second the Bifurcation Equation. We deal with the Auxiliary Equation first. 
\subsection{Step (2)}  $\ $\\
For 

$$\Psi:= \bigPi_Y \circ \Phi \;: \;U(0) \times \R^2 \times \R \times \R \longrightarrow Y$$ 

we have that 

$$\del_V \Psi\big(0, \alpha^b, \beta^b, 0, 0\big)\mathlarger{|}_X = \bigPi_Y \circ \partial_V \Phi\big(0, \alpha^b, \beta^b, 0, 0\big)|_X.$$

By  \eqref{positivity2}  the map

$$\del_V \Psi\big(0,\alpha^b, \beta^b, 0, 0\big)\mathlarger{|}_X : X \longrightarrow Y$$

is an isomorphism and hence through the implicit function theorem there exist $\;\varepsilon_i >0, i = 1,2,3,\;$ an open neighborhood $\;U_{\text{Moeb}}(0) \subset \text{ Moeb}_{f^b}T^2_b\;$ and a smooth function
 
\begin{equation} \label{tildeVV}
\tilde V : U_{\text{Moeb}}(0) \times ]-\varepsilon_1+ \alpha^b, \alpha^b + \varepsilon_1[ \times]-\varepsilon_2+ \beta^b, \beta^b + \varepsilon_2[ \times ]-\varepsilon_3,  \varepsilon_3[ \times ]-\varepsilon_4,  \varepsilon_4[
\\ \longrightarrow U(0)\; \cap\; X \subset W^{4,2, \perp}\big(T^2_b, S^3\big) 
\end{equation}

such that $\;V\big(m, \alpha, \beta, t, s\big) = m + \tilde V\big(m, \alpha, \beta, t, s\big)\;$ satisfies

$$\Psi\big(V\big(m, \alpha, \beta, t, s\big), \alpha, \beta, t, s\big) = 0$$

for all 

$$\big(m, \alpha, \beta, t, s\big) \in U_{\text{Moeb}}(0) \times ]-\varepsilon_1+ \alpha^b, \alpha^b + \varepsilon_1[ \times]-\varepsilon_2+ \beta^b, \beta^b + \varepsilon_2[ \times ]-\varepsilon_3,  \varepsilon_3[ \times ]-\varepsilon_4,  \varepsilon_4[.$$

 Further, these are the only solutions to 
 
 $$\;\Psi\big(V, \alpha, \beta, t, s\big) = 0\quad\text{ with }\quad V \in W^{4,2, \perp}\big(T^2_b, S^3\big)\;$$
 
  close to $\;0\;$ in the $\;W^{4,2}$-topology and $\;\alpha \sim \alpha^b,\;$ $\beta \sim \beta^b\;$ and $\;t,\;s\sim 0.\;$ By the definition of $\;\Psi\;$ we have classified all solutions of 

\begin{equation}\label{aux}
\bigPi_Y\Big( \delta \mathcal W_{\alpha, \beta} \big(\exp_{ f^b} \big( V + t\varphi^b+ s \tilde \varphi^b\big)\big)\Big) = \bigPi_Y \Big(\Phi\big(V,\alpha, \beta, t, s\big)\Big) = 0 
\end{equation}

with $\;V \in W^{4,2 \perp}\big(T^2_b, S^3\big)\;$ close to $\;0,\;$ $\alpha \sim \alpha^b, \;$ $\beta\sim \beta^b\;$ and $\;t,\;s \sim 0.$

\subsection{Step (3)}$\ $\\
We now turn to the bifurcation equation

$$\bigPi_{\text{Moeb}_{f^b}T^2_b \oplus <\varphi^b, \tilde \varphi^b>} \Phi\big(V, \alpha, \beta, t, s\big) = 0,$$

which we split into two equations

\begin{equation}\label{bifeq}
\begin{cases}
&\bigPi_{\text{Moeb}_{f^b}T^2_b}\Phi\big(V, \alpha, \beta, t, s\big) = 0\\
&\bigPi_{<\varphi^b, \tilde \varphi^b>} \Phi\big(V, \alpha, \beta, t, s\big) = 0.
\end{cases}
\end{equation}

The first equation has already been dealt with in \cite{NdiayeSchätzle1} (see Proposition B.2 and Equation (B.7)). The M\"obius invariance of $\;\mathcal W\;$ and $\;\bigPi\;$ implies that every solution of \eqref{aux} already solves the equation

$$\bigPi_{\text{Moeb}_{f^b}T^2_b}\Phi\big(V, \alpha, \beta, t, s\big) = 0,$$

for $\;V \in W^{4,2,\perp}\big(T^2_b, S^3\big) \;\text{ close to }\; 0\; \text{ and } \;\alpha \sim \alpha^b,\;\; \beta \sim \beta^b,\;\; t, \;s\sim 0.\;$ Let

$$f^b_{(t,s)} = \exp_{ f^b}\big(  t \varphi^b + s \tilde \varphi^b  \big)  $$ 

the family of surfaces considered in Proposition \ref{1d} by which there exist families of M\"obius transformations $\;M(t,s)\;$ and $\;\sigma(t,s)\;$ such that 

$$M(t,s) \circ f^b_{(t,s)}\circ \sigma(t,s) = \exp_{ f^b}\big( c(t,s) \varphi^b \big).$$ 

 Because $\;M(s,t)\;$ act on $\;S^3\;$ as isometries, we obtain for any solution of the Auxiliary Equation in Step (1) that

$$f(V(m, \alpha, \beta, t, s), t,s) = M(t,s) \circ \big( \exp_{ f^b}\big( V(m, \alpha, \beta, t, s) +  t \varphi^b + s \tilde \varphi^b \big)\circ \sigma(t,s) $$

 is given by 
 
$$f(V(m, \alpha, \beta, t, s), t,s) =  \exp_{ f^b}\big( \bar V(m, \alpha, \beta, t, s) + c(t,s) \varphi^b \big), $$
 
 with 
 
 $$\;\bar V(m, \alpha, \beta, t, s) = M(t,s) \circ V(m, \alpha, \beta, t, s)\circ \sigma(t,s) \;\perp\; <\varphi^b, \tilde \varphi^b>.$$
 
Therefore we can restrict ourselves without loss of generality to the equation

$$ \bigPi_{<\varphi^b>} \Phi\big(V, \alpha, \beta, t,0\big) = 0.$$

Note that this equation and the maps involved remain well-defined for $\;b \longrightarrow 1.\;$ Now the situation is very similar to the situation of bifurcation from simple eigenvalues. To abbreviate the notations let 

$$\Phi\big(V, \alpha, \beta, t\big) :=\Phi\big(V, \alpha, \beta, t, 0\big) \quad \text{ and }\quad V\big(m, \alpha, \beta, t\big):=V\big(m, \alpha, \beta, t, 0\big).$$

 We have derived that there exist a smooth function $V$ satisfying 
 
$$\bigPi_{Y \oplus \text{ Moeb}_{f^b}T^2_b}\;\Phi\big(V\big(m, \alpha, \beta, t\big), \alpha, \beta, t\big) = 0$$

for all $\;\big(m, \alpha, \beta, t\big) \in U_{\text{Moeb}}(0) \times ]-\varepsilon_1+ \alpha^b, \alpha^b + \varepsilon_1[ \times]-\varepsilon_2+ \beta^b, \beta^b + \varepsilon_2[ \times ]-\varepsilon_3,  \varepsilon_3[.\;$ It remains to solve

 $$\bigPi_{<\varphi^b>}\Phi\big(V\big(m, \alpha, \beta, t\big), \alpha, \beta, t\big) = 0,$$
 
 or equivalently 
 
  $$\Phi\big(V\big(m, \alpha, \beta, t\big), \alpha, \beta, t\big)\cdot \varphi^b = 0,$$
  
for $\big(m, \alpha, \beta, t\big) \in U_{\text{Moeb}}(0) \times ]-\varepsilon_1+ \alpha^b, \alpha^b + \varepsilon_1[ \times]-\varepsilon_2+ \beta^b, \beta^b + \varepsilon_2[ \times ]-\varepsilon_3,  \varepsilon_3[.$\\

For the smooth family of surfaces

$$f_t^b =  \exp_{ f^b}\big( V(m, \alpha, \beta, t) + t \varphi^b \big)$$

we observe

\begin{equation*}
\begin{split}
\Phi\big(0, \alpha^b, \beta^b, 0\big) \cdot \varphi^b &= \delta \mathcal W_{\alpha^b, \beta^b}\big(f^b \big)\big(\varphi^b\big) = 0\\
\del_t|_{t = 0} \Phi\big(0, \alpha^b, \beta^b, 0\big) \cdot \varphi^b & =  \delta^2 \mathcal W_{\alpha^b, \beta^b}\big(f^b \big) \big(\dot f^b, \;\varphi^b\big)=0\\
\del^2_{tt}|_{t = 0} \Phi\big(0, \alpha^b, \beta^b, 0\big) \cdot \varphi^b & =  \delta^3 \mathcal W_{\alpha^b, \beta^b}\big(f^b \big)\big(\dot f^b, \;\dot f^b,  \;\varphi^b\big) \\
\del^3_{ttt}|_{t = 0} \Phi\big(0, \alpha^b, \beta^b, 0\big)\cdot \varphi^b &= \delta^4 \mathcal W_{\alpha^b, \beta^b}\big(f^b \big) \big(\dot f^b,  \;\dot f^b,  \;\dot f^b,  \;\varphi^b\big) + \delta^3 \mathcal W_{\alpha^b, \; \beta^b}\big(f^b \big)\big(\ddot f^b,  \;\dot f^b,  \;\varphi^b\big),
\end{split}
\end{equation*}

where $\dot {()}$ denote the derivative with respect to $\;t\;$ at $\;t=0$ and $\;\ddot f^b := \nabla_{\dot f^b} \dot f^b,\;$ where $\nabla$ is te Levi-Civita connection of $S^3.$\\

\begin{Lem}
With the notations above  we have for $\;b\sim 1$

$$\del^2_{tt}|_{t = 0} \Phi\big(0, \alpha^b, \beta^b, 0\big) \cdot \varphi^b =  0 \ \text{and} \ \ \del^3_{ttt}|_{t = 0} \Phi\big(0, \alpha^b, \beta^b, 0\big)\cdot \varphi^b < 0.$$

\end{Lem}

\begin{proof}
The aim is to use Proposition \ref{n-thderivative} for the conclusion. For this it is necessary to identify $\;\dot f\;$ with $\;\varphi^b\;$ appropriately.  
 For $\;b \sim 1\;$  consider again
 
$$f^b_t= \exp_{ f^b}\big( V(m, \alpha, \beta, t) + t \varphi^b \big).$$

Then we have $\;\dot f^b =\varphi^b + \dot V(m, \alpha^b, \beta^b, 0).\;$ Since $\;V \in X\;$ we have that also $\;\dot V(m, \alpha^b, \beta^b, 0) \in X.\;$ Further, $\;f_b(t)\;$ solves the constrained Willmore equation on $\;X\;$ from which we obtain

$$\delta^2 \mathcal W_{\alpha^b, \beta^b}\big(f^b\big) \big( \dot V , \;V_0\big) = 0 \text{ for all } V_0 \in X.$$

From this we have $\;\dot V \in X^\perp\;$ and therefore $\;\dot V \in X \cap X^\perp\;$ and we obtain $\;\dot V=0\;$ and $\;\dot f^b = \varphi^b\;$ showing the assertion. \\

For the second derivative $\;\ddot f^b = \ddot V(m, \alpha, \beta, t)\;$ consider the candidates constructed in \cite{HelNdi2}. They are $\;W^{4,2}\;$ close to the homogenous torus and thus there exist maps $\;t(a,b)\;$ and $\;V\big(m(a,b), \alpha(a,b), \beta(a,b), t(a,b)\big)\;$ such that the candidate surfaces have the following representation:

$$f_{(a,b)} =  \exp_{ f^b}\Big( V\big(m(a,b), \alpha(a,b), \beta(a,b), t(a,b) \big)+ t(a,b) \varphi^b\Big)$$

Since $\;\big(\del_{\tilde a} f_{(a,b)}|_{\tilde a=0}\big)^\perp = \varphi^b\;$, we have that 

$$\del_{\tilde a} t(a,b)|_{\tilde a=0} = 1 $$

 and 
 
 $$ \del_{\tilde a} V|_{\tilde a= 0} = \Big(\del_{\tilde a} \alpha(a,b) \del_{\alpha} V + \del_{\tilde a} \beta(a,b) \del_{\beta} V + \del_{\tilde a} t(a,b) \del_{t} V\Big)|_{\tilde a= 0} = 0.$$

For $\;b \longrightarrow 1\;$ we obtain with similar arguments as for $\;\dot V\;$ that 

$$\delta^2 \mathcal W_{\alpha^1, \beta^1}\big(f^1\big) \big(\del_{\alpha}|_{\alpha= \alpha^b} V (\alpha^1, \beta^1, 0),\; \cdot\big)  =  \delta \bigPi^1\big(f^1\big)\big(\cdot \big) = 0$$

from which we obtain that $\;\del_\alpha V|_{\alpha = \alpha^b} = 0.\;$ Further, 

$$\del^2_{\tilde a \tilde a} \beta(a,b)|_{\tilde a= 0} = \del_a \beta(a,b)|_{a= 0} = \del_b \alpha^b|_{b = 1} = 0. $$

The last equality is due to the fact that  $\;\alpha^{b} = \alpha^{\tfrac{1}{b}}.\;$
Moreover, we have already computed that $\;\del_{\tilde a} \alpha(a, b) = \del_{\tilde a} \beta(a,b) = 0.\;$
For the second derivative $\;\del^2_{\tilde a \tilde a}f_{(a,b)} := \nabla_{\del_{\tilde a}f_{(a,b)}}\del_{\tilde a}f_{(a,b)}\;$ we thus obtain

\begin{equation*}
\begin{split}
\lim_{b \rightarrow1}\del^2_{\tilde a \tilde a} f_{(a,b)}|_{\tilde a = 0}&=\lim_{b \rightarrow1}\del_{\tilde a}\varphi(a)|_{a=0} \\
 &= \lim_{b \rightarrow1}\del^2_{\tilde a\tilde a}  t(a,b)|_{\tilde a = 0} \varphi^b + \lim_{b \rightarrow1}\big(\del_{\tilde a} t(a,b)\big)^2|_{\tilde a=0} \ddot V_b \big(m, \alpha^b, \beta^b, 0\big)\\
&= \lim_{b \rightarrow1}\del^2_{\tilde a\tilde a}  t(a,b)|_{\tilde a = 0} \varphi^b + \lim_{b \rightarrow1}\ddot V \big(m, \alpha^b, \beta^b, 0\big)\\
& = \del^2_{\tilde a\tilde a}  t(a,1)|_{\tilde a = 0} \varphi^1 + \ddot f^1.
\end{split}
\end{equation*}

By the first assertion of Proposition \ref{n-thderivative} we thus obtain

$$\delta^3 \mathcal W_{\alpha^1, \beta^1} \big (\del_{\tilde a}\varphi(a)|_{a=0, b=1}, \varphi^1, \varphi^1 \big) = \delta^3 \mathcal W_{\alpha^1, \beta^1} \big(\ddot f^1, \varphi^1, \varphi^1\big)$$

and therefore

$$\lim_{b \rightarrow 1}\partial^3_{ttt} \Phi\big(0, \alpha^b, \beta^b, t \big) \cdot \varphi^b < 0.$$

By continuity we get that this remains true for $\;b \sim 1\;$ close enough.

\end{proof} 

Now we can use classical arguments in bifurcation theory (bifurcation from simple eigenvalues) to obtain a unique function $\;t(m, \alpha, \beta)\sim 0\;$  satisfying

$$ \Phi\Big(V\big(m, \alpha, \beta, t(m, \alpha, \beta)\big), \alpha, \beta, t(m, \alpha, \beta)\Big) \cdot \varphi^b = 0.$$

Moreover, all solutions to 
 
 $$\Phi\big(V(m, \alpha, \beta, t), \alpha, \beta, t \big) \cdot \varphi^b = 0$$
 
 for $\;(m, \alpha, \beta, t) \in U_{\text{Moeb}}(0) \times ]-\varepsilon_1+ \alpha^b, \alpha^b + \varepsilon_1[ \times]-\varepsilon_2+ \beta^b, \beta^b + \varepsilon_2[ \times ]-\varepsilon_3,  \varepsilon_3[\;$ are of this form for sufficiently small $\;\varepsilon_i\;$ and $\;U_{\text{Moeb}}(0).\;$ In other words, 
 
$$f_{\alpha, \beta}^{m}:= \exp_{ f^b}\Big( V\big(m, \alpha, \beta, t(m, \alpha, \beta)\big) + t(m, \alpha, \beta) \varphi^b\Big)$$

are the only solutions to 

\begin{equation}\label{*}
\begin{split}
&\delta \mathcal W_{\alpha, \beta} \big(f \big) = 0 \quad \text{ with}\\
&f = \exp_{ f^b}\Big( W^{4,2 \perp}\big(T^2_b, S^3\big) \;\cap\; <\tilde \varphi^b>^\perp \Big)
\end{split}
\end{equation}

which are $\;W^{4,2}$-close to $\;f^b\;$ $\alpha \sim \alpha^b,\;$ and $\;\beta \sim \beta^b.\;$ For fixed $\;(\alpha, \beta) \sim (\alpha^b, \beta^b)\;$ we thus obtain a manifold worth of solutions of dimension $\;\dim \big(\text{Moeb}_{f^b}T^2_b\big) + 1.$\\

Since $\;\mathcal W\;$ and $\;\bigPi\;$ is M\"obius and parametrization invariant, we get for any M\"obius transformation $\;M\;$ with

$$M \circ \exp_{ f^b}\Big(  V\big(m ,\alpha, \beta, t(m,\alpha, \beta)\big) + t(m ,\alpha, \beta)\varphi^b\Big)\subset S^3$$

and every

$$\sigma \in \text{Diff} = \text{Diff}_{T^2_b} := \big\{\;\psi : T^2_b \longrightarrow T^2_b \; \; | \; \; \psi  \;\text{ is a smooth diffeomorphism }\big\}$$

that the following equation holds

$$\delta \mathcal W_{\alpha, \beta}\left (M \circ \exp_{ f^b}\left( V\big(m ,\alpha, \beta, t(m,\alpha, \beta)\big) + t(m ,\alpha, \beta)\varphi^b\right) \circ \sigma \right) = 0.$$

The M\"obius group $\;$Moeb$(3)\;$ of $\;S^3\;$ is a finite dimensional Lie group and for an appropriate neighborhood $\;U($Id$)\; \subset \text{Moeb}(3)\;$ and $\;(\alpha, \beta) \in\; ]-\varepsilon_1 + \alpha^b, \alpha^b+ \varepsilon_1[ \times ]-\varepsilon_2 + \beta^b, \beta^b+ \varepsilon_2[\;$ we have

$$M \circ \exp_{ f^b}\left( V\big(m ,\alpha, \beta, t(m,\alpha, \beta)\big) + t(m ,\alpha, \beta)\varphi^b \right)$$

 is $\;\mathcal C^1$-close to $\;f^b\;$ and hence we can write 

$$M \circ\exp_{ f^b}\Big( V\big(m ,\alpha, \beta, t(m,\alpha, \beta)\big) + t(m ,\alpha, \beta)\varphi^b\Big) \circ \sigma =\exp_{f^b}( W)$$

for an appropriate $\;W \in W^{4,2, \perp}(T^2_b,S^3)\;$ and $\;\sigma \in \text{Diff}.\;$ More precisely, for the nearest point projection 

$$\bigPi_{f^b}: U_{\delta} := \big\{\; x \in S^3 \; | \ \text{dist}\big(x, f^b(T^2_b)\big) < \delta \;\big\} \longrightarrow f^b\big(T^2_b\big)$$

for an appropriate small positive $\;\delta,\;$ we have

$$\sigma := \sigma (M, \alpha, \beta) = \big(f^b\big)^{-1}\circ \bigPi_{f^b}\circ  M \circ f^b$$

and 

$$W = W(M, \alpha, \beta) = M \circ\Big(V\big(m ,\alpha, \beta, t(m,\alpha, \beta)\big) + t(m ,\alpha, \beta)\varphi^b\Big) \circ \sigma.$$

Now since $\;f^{m}_{\alpha, \beta}\;$ are the only solutions to \eqref{*} in $\; \exp_{f^b} \big(W^{4,2,\perp} (T^2_b, S^3)\;\cap <\tilde \varphi>^\perp\big)\;$ which are $\;W^{4,2}$-close to $\;f^b\;$  we get

\begin{equation*}
\begin{split}
W(M, \alpha, \beta) &= V\big(m, \alpha, \beta, t(m, \alpha, \beta)\big) + t (m , \alpha, \beta)\varphi^b \\
&= m + \tilde V\big(m, \alpha, \beta, t(m, \alpha, \beta)\big) + t(m , \alpha, \beta)\varphi^b
\end{split}
\end{equation*}

for some $\;m \in U_{\text{Moeb}}(0) \subset \text{ Moeb}_{f^b}(T^2_b).\;$ More precisely we have

$$m:=  m(M, \alpha, \beta) = \bigPi_{\text{Moeb}_{f^b}T^2_b} W\big(M, \alpha, \beta\big).$$

Since $\;\tilde V\;$ is a smooth map into $\;W^{4,2}\big(T^2_b, S^3\big) \subset \mathcal C^2\big(T^2_b, S^3\big)\;$ we obtain that the maps 

$$\big(M, \alpha, \beta\big) \longmapsto \sigma, \;W, \;m$$

are continuously differentiable into $\;C^1\big(T^2_b, S^3\big).\;$ Hence we obtain for $\;\chi \in T_{Id}$ Moeb$(3)$   

$$\del_{M} W\big(M, \alpha^b, \beta^b\big)\cdot \chi\mathlarger{|}_{M= Id} = \big(\chi \circ f^b\big) \; +  \;df^b  \big(\del_M\sigma(M, \alpha^b, \beta^b)\mathlarger{|}_{M= Id}\cdot \chi\big) = P_{\chi \circ f^b}\big(\chi \circ f^b\big) \in\text{ Moeb}_{f^b}T^2_b$$

and thus

$$\del_{M}m\big(Id, \alpha^b, \beta^b\big) \cdot \chi = \bigPi_{\text{Moeb}_{f^b}T^2_b}\Big( P_{\chi \circ f^b}\big(\chi \circ f^b\big)\Big) = P_{\chi \circ f^b}\big(\chi \circ f^b\big).$$

By definition of $\;$Moeb$_{f^b}T^2_b\;$ we thus obtain that

 $$\del_M m\big(Id, \alpha^b, \beta^b\big) :  T_{Id}\text{Moeb}(3) \longrightarrow \text{Moeb}_{f^b} T^2_b$$
 
 is surjective and hence by implicit function theorem and $\;m\big(Id, \alpha^b, \beta^b\big) = 0\;$ we have 

$$\tilde U_{\text{Moeb}}(0) \subset m\Big(U(0) \times \big\{(\alpha, \beta)\big\}\Big) $$

for some  open neighborhood $\;\tilde U_{\text{Moeb}}(0)\;$ of $\;0\;$ in $\;$ Moeb$_{f^b} T^2_b\;$ independent of $\;(\alpha, \beta).\;$ Therefore we have that

\begin{equation}\label{xx}
M \circ \exp_{ f^b}\Big(V\big(m ,\alpha, \beta, t(m,\alpha, \beta)\big) + t(m ,\alpha, \beta)\varphi^b\Big) \circ \sigma
\end{equation} 

are the only solutions to \eqref{*} which are $\;W^{4,2}$-close to $\;f^b.$

\subsection{Step (4)}$\ $\\The aim is to identify the Teichm\"uller class of the solutions of \eqref{*} given by \eqref{xx} for fixed $\;b \sim 1\;$ and $\;b \neq1.\;$ In particular, we show that the solutions of \eqref{*} induces a local diffeomorphism between the space of Lagrange multipliers (around $\;(\alpha^b, \beta^b)$) to the Teichm\"uller space of tori around the class of the Clifford torus $\;(0,1) \in \H^2.\;$ 
Clearly, by setting 

\begin{equation}\label{VbV}
V_b = V\big(\alpha, \beta, t(0,\alpha, \beta)\big) + t(0, \alpha, \beta)\varphi^b
\end{equation}

we have 

$$\bigPi\left(M \circ \exp_{ f^b}\big( V(\alpha, \beta)\big) \circ \sigma\right)^*g_{S^3} = \bigPi\left(\exp_{ f^b}\left(V(\alpha, \beta)\right)\right)^*g_{S^3}.$$

Thus for all solutions of \eqref{*} we have that

\begin{equation*}
\begin{split}
\begin{pmatrix}c (\alpha, \beta) \\ d(\alpha, \beta)\end{pmatrix} &= \begin{pmatrix} \bigPi^1\exp_{ f^b}\left(V(\alpha, \beta)\right)^*g_{S^3} \\ \bigPi^2\exp_{ f^b}\left(V(\alpha, \beta)\right)^*g_{S^3}\end{pmatrix}\\
&= \begin{pmatrix} \bigPi^1\exp_{ f^b}\left(V(\alpha, \beta)\right)^*g_{S^3} \\ \bigPi^2\exp_{ f^b}\left(V(\alpha, \beta)\right)^*g_{S^3} \end{pmatrix}
\end{split}
\end{equation*}

independently of $\;m \in \tilde U_{\text{Moeb}}(0).\;$ We first solve for $\;\bigPi^2,\;$ i.e., want to solve the equation 

$$d(\alpha, \beta) =  \tilde b \quad \text{ for } \quad \tilde b \sim b.$$

By definition we have

$$d(\alpha^b, \beta^b) = b$$

and further

$$\del_{\beta}\mathlarger{|}_{\beta= \beta^b} d(\alpha^b, \beta^b) = \delta \bigPi^2\big(f^b\big)\big(\del_{\beta}\mathlarger{|}_{\beta = \beta^b} V_b(\alpha^b, \beta)\big).$$

Then from 

$$\Phi\big(V(\alpha, \beta), \alpha, \beta, t(0, \alpha, \beta)\big) = 0$$

with $\;V(\alpha, \beta) := \tilde V\big(0, \alpha, \beta, t(0, \alpha, \beta)\big)\;$
and 

$$\del_V \Phi\big(0, \alpha^b, \beta^b, 0\big) \cdot Z = \delta^2 \mathcal W_{\alpha^b, \beta^b}\big(f^b\big)\big(Z,\;.\big)$$

we derive that

$$\del_V \Phi\big(0, \alpha^b, \beta^b, 0\big) \cdot \del_\beta\mathlarger{|}_{\beta= \beta^b} V\big(\alpha^b, \beta^b\big) + \del_\beta\mathlarger{|}_{\beta= \beta^b} \Phi\big(0, \alpha^b, \beta^b, 0\big) + \del_t \Phi\big(0, \alpha^b, \beta^b, 0\big) \del_\beta\mathlarger{|}_{\beta= \beta^b}t\big(0, \alpha^b, \beta^b\big) = 0.$$

Thus we get

\begin{equation}
\begin{split}\delta^2\mathcal W_{\alpha^b, \beta^b}\big(f^b\big)\big(\del_\beta|_{\beta  = \beta^b} V(\alpha^b, \beta^b), \;.\big) &- \delta \bigPi^2\big(f^b\big) + \delta^2\mathcal W_{\alpha^b, \beta^b}\big(f^b\big)\big( \del_\beta|_{\beta  = \beta^b} t(0, \alpha^b, \beta^b) \varphi^b, \;.\big) =0 \\
\Leftrightarrow \quad  \quad \delta^2\mathcal W_{\alpha^b, \beta^b}\big(f^b\big)\big(\del_\beta|_{\beta  = \beta^b} V(\alpha^b, \beta^b),\; .\big) &= \delta \bigPi^2\big(f^b\big).
\end{split}
\end{equation}

On the other hand, there exist a $\;V^0_b \in  C^\infty \big(T^2_b, S^3\big)\;$ such that $\;\delta \bigPi^2_{f^b} \big(V^0_b\big) \neq 0\;$ by Proposition 3.2. of \cite{NdiayeSchätzle1}. This implies

$$\delta^2\mathcal W_{\alpha^b, \beta^b}\big(f^b\big)\big(\del_\beta|_{\beta  = \beta^b} V(\alpha^b, \beta^b), V^0_b\big) \neq 0$$

 therefore $\;\del_\beta|_{\beta = \beta^b} V(\alpha^b, \beta^b) \notin \text{Moeb}_{f^b}T^2_b\; \oplus <\varphi^b, \tilde \varphi^b>\;$ and 

$$\delta^2\mathcal W_{\alpha^b, \beta^b}\big(f^b\big)\big(\del_\beta|_{\beta  = \beta^b} V(\alpha^b, \beta^b), \;\del_\beta|_{\beta  = \beta^b} V(\alpha^b, \beta^b) \big)>0$$

by \eqref{positivity2} or the computations in Section \ref{stability}. Hence using the implicit function theorem we have for $\;\alpha \sim \alpha^b\;$ and $\;\tilde b \sim b\;$ a unique $\;\beta(\alpha, d) \sim \beta^b$\; such that

$$d \big(\alpha, \beta(\alpha, \tilde b) \big) = \tilde b \quad \text{and} \quad \beta(\alpha^b, b)  = \beta^b$$

and the map $\;(\alpha, \tilde b) \longrightarrow \beta(\alpha, \tilde b) \;$ is smooth. In particular, If $\;\bigPi^2(f)= b\;$ and $\;\alpha(f)= \alpha^b\;$
 we obtain $\;\beta(f) = \beta^b.\;$ It remains to determine $\;\bigPi^1\;$ of the solutions of \eqref{*} given in \eqref{xx}. The equation we aim to solve is 

$$c\big(\alpha(\tilde b, \beta), \;\beta\big) = a \quad \text{with} \quad a\sim 0.$$

We have 

\begin{equation}
\begin{split}
&c\big(\alpha^b, \;\beta(\alpha^b, b)\big) =0\\
\del_{\alpha}\mathlarger{|}_{\alpha = \alpha^b}\Big[&c\big(\alpha, \;\beta(\alpha, b)\big)\Big] = \delta \bigPi^1\big(f^b\big) \cdot \Big(\del_{\alpha}\mathlarger{|}_{\alpha= \alpha^b} \big[V_b\big(\alpha,\; \beta(\alpha, b)\big)\big]\Big) = 0\\
\del^2_{\alpha}\mathlarger{|}_{\alpha= \alpha^b}\Big[&c\big(\alpha,\; \beta(\alpha, b)\big)\Big] = \delta^2 \bigPi^1\big(f^b\big) \Big(\del_{\alpha}\mathlarger{|}_{\alpha= \alpha^b}\big[V_b\big(\alpha,\; \beta(\alpha, b)\big)\big], \del_{\alpha}\mathlarger{|}_{\alpha= \alpha^b} \big[V_b\big(\alpha,\; \beta(\alpha, b)\big)\big] \Big).
\end{split}
\end{equation}

Now, using the fact that 

$$\Phi\Big(V\big(\alpha, \;\beta(\alpha, b)\big),\; \alpha,\; \beta(\alpha, b), \; t\big(0, \alpha, \beta(\alpha,  b)\big)\Big)=0$$

we get

\begin{equation}
\begin{split}
\del_{\alpha}\mathlarger{|}_{\alpha= \alpha^b}\Phi\big(0, \alpha^b, \beta^b, 0\big) + \del_V\Phi\big(0, \alpha^b, \beta^b, 0\big)\cdot \Big(\del_{\alpha}\mathlarger{|}_{\alpha= \alpha^b} \big[V_b\big(\alpha, \beta(\alpha,  b)\big)\big]\Big)\\
 + \del_t  \Phi\big(0, \alpha^b, \beta^b, 0\big) \cdot \Big(\del_{\alpha}\mathlarger{|}_{\alpha= \alpha^b}\big[t\big(0,\alpha, \beta(\alpha,  b)\big)\big] \varphi^b\Big)  = 0.
\end{split}
\end{equation}

Thus we obtain $\;\Big($using $V_b \big(\alpha, \beta(\alpha, b)\big) = V\big(\alpha, \beta(\alpha,  b)\big) + t\big(0, \alpha, \beta(\alpha,  b)\big)\varphi^b\;$, see \eqref{VbV},  and  the fact that $\;V\big(\alpha, \beta(\alpha,  b)\big) \perp \;<\varphi^b, \tilde \varphi^b>\;$ by definition \eqref{tildeVV}$\Big)\;$

$$-2\del_{\alpha}\mathlarger{|}_{\alpha= \alpha^b}\delta\bigPi^1\big(f^b\big) + \delta^2 \mathcal W_{\alpha^b, \beta^b}\big(f^b\big) \Big(\del_{\alpha}\mathlarger{\mathlarger{|}}_{\alpha= \alpha^b}\big[V_b\big(\alpha, \beta(\alpha, b)\big)\big],\; .\Big) = 0$$

 and therefore we have
 
 \begin{equation}\label{xxx}
  \delta^2 \mathcal W_{\alpha^b, \beta^b}\big(f^b\big) \Big(\del_{\alpha}\mathlarger{\mathlarger{|}}_{\alpha= \alpha^b}\big[V_b\big(\alpha, \beta(\alpha, b)\big)\big], .\Big) = 0,
  \end{equation}
  
  which means that
  
 $$ \del_{\alpha}\mathlarger{\mathlarger{|}}_{\alpha= \alpha^b} \Big[V_b\big(\alpha, \beta(\alpha, b)\big)\Big] \in \text{ Moeb}_{f^b}T^2 \oplus <\varphi^b, \tilde \varphi^b>,$$
 
 i.e., by  \eqref{VbV} and \eqref{tildeVV}
 
 $$ \del_{\alpha}\mathlarger{\mathlarger{|}}_{\alpha= \alpha^b}\Big[V\big(\alpha, \beta(\alpha, b)\big)\Big] \in \text{ Moeb}_{f^b}T^2.$$
 
 Therefore, we get by \eqref{xxx} 
  
  \begin{equation}
  \begin{split}
  \alpha^b\delta^2 \bigPi^1\big(f^b\big) &\Big( \del_{\alpha}\mathlarger{\mathlarger{|}}_{\alpha= \alpha^b} \big[V_b\big(\alpha, \beta(\alpha, b)\big)\big],  \del_{\alpha}\mathlarger{\mathlarger{|}}_{\alpha= \alpha^b} \big[V_b\big(\alpha, \beta(\alpha, b)\big)\big]\Big)\\
   &= \delta^2 \mathcal W_{\beta^b}\big(f^b\big) \big(\varphi^b, \varphi^b\big) \Big(\del_{\alpha}\mathlarger{\mathlarger{|}}_{\alpha= \alpha^b} \big[t\big(0,\alpha, \beta(\alpha, b)\big)\big]\Big)^2.
   \end{split}
   \end{equation}
   
Using $\;\del_{\alpha}\mathlarger{\mathlarger{|}}_{\alpha= \alpha^b} \Big[t\big(0,\alpha, \beta(\alpha, b)\big)\Big] \neq 0\;$ this implies that

$$\del^2_{\alpha}\mathlarger{\mathlarger{|}}_{\alpha= \alpha^b} \Big[c\big(\alpha, \beta(\alpha, b)\big)\Big] > 0.$$

Hence as above, using classical arguments in bifurcation theory via monotonicity we have that there exist a unique branch of solutions 
$\;\alpha(a, \tilde b)\;$ such that 

$$c\Big(\alpha(a, \tilde b), \;\beta\big(\alpha (a,\tilde b), \tilde b\big)\Big) = a$$

for $\;a \sim 0^+\;$ and $\;\tilde b \sim b\;$ with $\;\alpha(0, b) = \alpha^b, \;$ and $\;\alpha(a,b) \leq \alpha^b.\;$
Altogether we obtain for  $\;b \sim 1\;$ but $\; b\neq1\;$ fixed, a family of smooth solutions to (up to invariance)

$$\delta \mathcal W_{\alpha, \beta}\big(f\big) = 0, \quad \alpha \sim \alpha^b, \quad \alpha \leq \alpha^b\quad \text{ and } \quad \beta \sim \beta^b$$ 

parametrized by their conformal type $\;a \sim 0^+,\;$ such that the only solution with 
$\;\alpha= \alpha^b\;$ and  $\;\bigPi^2(f)= b\;$ is the homogenous torus of conformal class $(0,b).$

\end{proof}

\section{Reduction of the global problem to a local one}
We use  penalization and relaxation techniques of Calculus of Variations to establish Theorem \ref{minimizers} providing the existence of appropriate global minimizers in an open neighborhood of each rectangular class close to the square class. By appropriate global minimizer we mean those reducing our clearly global problem to a local problem, i.e., which are close to the Clifford torus in $\;W^{4,2}\;$ with prescribed behavior of its Lagrange multipliers. Then Theorem \ref{cwm} shows that these abstract minimizers coincides with the candidate surfaces. \\

\begin{The}\label{minimizers}
For every $\;b \sim 1\;$ there exists an $\;a^b \;$ small with the property that 
 for all $\;a \in [0, a^b]\;$ the infimum of Willmore energy 
 
$$\Min_{(a,b)} = \inf \;\Big\{\; \mathcal W_{\alpha^b}\big(f\big)\; \mathlarger{\mathlarger{|}}\   f: T^2_b \longrightarrow S^3 \text{ smooth immersion } \mathlarger{|} \  0 \leq \bigPi^1\big(f\big)  \leq a \text{ and } \bigPi^2\big(f\big) = b\;\Big\}$$

is attained by a smooth immersion $\;f^{(a,b)}: T^2_b \longrightarrow S^3\;$ of conformal type $\;(a,b)\;$ and verifying 

$$\delta \mathcal W_{\alpha^{(a,b)}, \;\beta^{(a,b)}}\big(f^{(a,b)}\big) = 0$$ 

with  $\;\alpha^{(a,b)}\leq \alpha^b$ and $\alpha^{(a,b)} \longrightarrow \alpha^b\;$ almost everywhere as $\;a \longrightarrow 0\;$
 and $\;\beta^{(a,b)} \longrightarrow  \beta^b,\;$ as $\; a \longrightarrow 0\;$ where $\;(\alpha^b, \beta^b)\in \R^2\;$ as defined in Theorem \ref{cwm}.
\end{The}
\begin{proof} 
By taking $\;b\sim 1\;$ close enough, we have that (using the same arguments as in \cite{NdiayeSchätzle1}, existence part) there exists  $\;a^b >0\;$ small with the property that for all $\;a \in [0, a^b]\;$ the minimization problem 

$$\text{Min}_{(a,b)} = \inf \;\Big\{\; \mathcal W_{\alpha^b}\big(f\big) \;\mathlarger{|}\   f: T^2_b \longrightarrow S^3 \;\text{ smooth immersion } \mathlarger{|} \ 0 \leq \bigPi^1\big(f\big)  \leq a \;\text{ and } \;\bigPi^2\big(f\big) = b \;\Big\}$$

is attained by a smooth immersion $\;f^{(\tilde a,b)}_a\;$ with conformal type $\;(\tilde a,b)\;$ and $\;\tilde a \in [0, a]\;$ solving the Euler-Lagrange equation for (conformally) constrained Willmore tori

$$\delta \mathcal W_{\alpha^{(\tilde a,b)}_a, \;\beta^{(\tilde a,b)}_a}\Big(f^{(\tilde a,b)}_{a}\Big)= 0$$

for some $\;\alpha^{(\tilde a,b)}_a,\; \beta^{(\tilde a,b)}_a \in \R.$

\subsection*{Step (1): $\mathbf{\tilde a =a}$}$\ $\\
For $\;a = 0\;$ the homogenous tori $\;f^b\;$ are the unique minimizer and $\;\tilde a=a=0.\;$ Thus let  $\;a>0\;$ in the following. The candidate surfaces $\;f_{(a,b)}\;$ with $\;f_{(0,b)} = f^b\;$ constructed in \cite{HelNdi2}  satisfy that 

 $$\mathcal W_{\alpha^b}\big(f_{(a,b)}\big) = \omega_{\alpha^b}(a,b)$$ 
 
 is strictly decreasing for $\;a\sim 0,\;$ since 
 
 $$\frac{\del \omega_{\alpha^b}(a,b)}{\del a} = \alpha_{(a,b)} - \alpha^b <0.$$
 
This yields $\;\tilde a >0$.\\

Now, we claim that up to take $\;a^b\;$ smaller $\;\tilde a = a\;$ holds for all $\;a \;\in\; ]0, a^b].\;$  Assume this is not true. Then since $\;\tilde a >0\;$ there would exist a sequence $\;a_n \longrightarrow 0\;$ with corresponding $\;\tilde a_n \longrightarrow 0\;$ sucht that

$$\alpha_n^b\;:=\; \alpha_{\tilde a_n}^{a_n, b}\; =\; \alpha^b \quad\quad\forall n.$$

Then arguing as in \cite{NdiayeSchätzle1} gives

\begin{equation}\label{convergencef}
f^b_n\;:=\;f^{a_n,b}_{\tilde {a}_n} \longrightarrow f^b \quad\quad \text{smoothly} 
\end{equation}

up to invariance. This is a contradiction to our classification of solutions around $\;f^b, \;$ because $\;f^b_n = f^b\;$ implies $\;\tilde a_n = 0, \;$ while we have $\;\tilde a_n>0.$\\

\begin{Rem}\label{monotonevarphi}
Because the minimum $\;\Min_{(a,b)}\;$ for $\;a \in (0,a^b)\;$ is always attained at the boundary, the function 

$$\varphi(a,b):=  \Min_{(a,b)} = \omega(a,b) - \alpha^b a,$$ 

where $\;\omega(a,b)\;$ is the minimal Willmore energy in the class $(a,b),$ is monotonically non-increasing. Therefore $\;\varphi(a,b)\;$ (and thus also $\;\omega(a,b)$) is differentiable almost everywhere in $\;a\;$ and $$\;\del_a\varphi(a,b) \leq 0, \;$$   almost everywhere. \\

\end{Rem} 
\subsection*{Step (2): $\mathbf{\frac{\del}{\del a} \omega(a) = \alpha^{(a,b)} \leq \alpha^b}$ almost everywhere} $\ $\\
The aim in this step is to show the first statement (1) of Lemma \ref{mainobservation} with weaker regularity assumptions on the dependence of $\;f^{(a,b)}\;$ on its conformal class, i.e., to relate $\;\frac{\del \omega(a)}{\del a}\;$ with $\;\alpha^{(a,b)}\;$ for almost every $\;a \in ]0, a^b[\;.$ Then by Remark \ref{monotonevarphi} we obtain the claimed upper bound on the Lagrange multipliers $\;\alpha^{(a,b)}.$\\

For $\;b \sim 1\;$ fixed we can assume up to taking $\;a^b\;$ smaller and by the same arguments in step 1 that the minimizers $\;f^{(a,b)}\;$ are non-degenerate for all $\;a\in (0, a^b).\;$  For $\;a_0 \in  (0, a^b)\;$ such that $\;\omega(a,b)\;$ is differentiable 
 choose variational vector fields $\;V_i^{(a_0, b)}\;$ satisfying $\;\delta \bigPi^i\big(V_j^{(a_0,b)}\big) = \delta_{i,j}\;$ and consider the smooth family of immersions
 
$$\bar f(s,t) := \exp_{f^{(a_0,b)}}\left( tV_1^{(a_0,b)} + sV_2^{(a_0,b)}\right).$$

Then solving the equation 

$$\bigPi\big(\bar f(s,t) \big) = (a,b)$$

defines unique maps $\;t(a)\;$ and $\;s(a)\;$ (with $\;t(a_0) = 0\;$ and $\;s(a_0) = 0$) by the implicit function theorem,
since 

$$\det \Big(\delta \bigPi^i\big(V_j^{(a_0,b)}\big)\Big)_{i,j = 1,2} = 1.$$

Further, consider the Willmore energy of this family $\;\bar f\big(s(a),\;t(a)\big)$ 

$$\bar \omega(a,b) := \mathcal W \Big(\bar f\big(s(a), \;t(a)\big)\Big).$$

Then we can compute

$$\del_a \bar \omega(a,b)\mathlarger{|}_{a= a_0} = \alpha^{(a_0,b)}\del_a\mathlarger{|}_{a=a_0} t(a)$$

Observe that $\;t(a)\;$ and $\;s(a)\;$ are smooth in $a$ and the Taylor expansion for $\;a\;$ and $\;b\;$ gives

$$a= a_0 + t(a) + o\big(\mathlarger{|}t(a)\mathlarger{|}\big).$$

Therefore $\;\del_a\mathlarger{|}_{a= a_0} t(a) = 1\;$ and thus 

$$\del_a \bar \omega(a,b)\mathlarger{|}_{a= a_0} = \alpha^{(a_0,b)}.$$

Now, comparing $\;\bar\omega(a,b)\;$ to  $\;\omega(a,b)$ -- the minimal Willmore energy in the conformal class $\;(a,b)\;$ we obtain that the function 

$$\Delta(a) =\bar\omega(a,b) - \omega(a, b)\geq 0$$ 

with equality at $\;a=a_0.\;$ In other words $\;\Delta\;$ has a local minimum at $\;a= a_0.\;$  Because $\;\omega(a,b)\;$ is differentiable at $\;a= a_0\;$ by assumption and $\;\bar \omega(a,b)\;$ is smooth, we have $\;\del_{a} \Delta\mathlarger{|}_{a = a_0} = 0. \;$ This gives

 $$ \del_{a} \omega(a, b)\mathlarger{|}_{a= a_0} = \del_{a} \bar\omega(a,b)\mathlarger{|}_{a= a_0}  = \alpha^{(a_0,b)}.$$

\subsection*{Step(3): $\mathbf{\lim_{a\rightarrow 0,  a.e.} \alpha^{(a,b)} = \alpha^b}$} $\ $\\

Since $\;\delta\bigPi^2\big(f^b\big)\neq 0,$ we obtain

$$\lim_{a\rightarrow 0}\beta(a, b) \longrightarrow \beta^b,$$

by standard weak compactness argument. Thus it is only necessary to show the convergence of $\;\alpha^{(a,b)}.\;$ We will show its convergence for $\;a \longrightarrow 0\;$ almost everywhere, by which we mean the convergence up to a zero set $\;A \subset [0, a^b),\;$ i.e., 

$$\lim_{a\;\rightarrow \;0,\  a.\;e.} \alpha^{(a,b)}\;:= \lim_{a\;\rightarrow\; 0, \;a \;\in \;[0, a^b)\setminus A}\; \alpha^{(a,b)}.$$

We first show that 

$$\alpha_{sup}:= \limsup_{a\; \rightarrow \;0,\  a.\;e.} \alpha^{(a,b)} = \alpha^b.$$

Clearly, $\;\alpha_{sup} \geq 0.\;$ Otherwise, $\;\del_a \omega(a,b) < 0\;$ almost everywhere. Because of the monotonicity of $\omega(a,b) - a\alpha^b$ and the continuity of $\;\omega(a,b),\;$ we would thus obtain that  $\;\omega(a,b)\;$ is decreasing in $a$ contradicting the fact that $\;\omega(0,b)\;$ is the minimum of $\;\omega(a,b)\;$ for $\;b \sim1.$\\

Assume now that $\;\alpha_{sup} < \alpha^b.\;$ Then there exist a zero sequence $\;(a_k)_{k \in \N}\;$ with $\;a_k>0 \;$ such that the  Lagrange multipliers $\;\alpha^{(a_k, b)}\;$ converge to $\;\alpha_{sup} <\alpha^b.\;$ Thus the corresponding immersions $\;f^{(a_k,b)} \longrightarrow f^b\;$ smoothly up to invariance using same arguments as in step (1) to prove \eqref{convergencef}. But by Lemma \ref{alpha<alpha^b} we then obtain $\;f^{(a_k,b)}  = f^{(0,b)}\;$ for $\;k>>1\;$ in contradition to $\;a_k >0.$\\

Now we want to show that also 

$$\alpha_{inf}:=\liminf_{a\;\rightarrow \;0,\ a.\;e.} \alpha^{(a,b)}= \alpha_{sup} = \alpha^b.$$

For this we first show that $\;\alpha_{inf}\;$ is bounded from below, more precisely, $ \;\alpha_{inf} \geq 0.$\\

Up to choosing  $a^b$ smaller we have by the same arguments as above that $\;\alpha^{(a,b)} \neq 0\;$ for all $\;a \in [0, a^b).\;$ Assume that $\;\alpha_{inf} <0.\;$ Then, since $\;\alpha_{sup} = \alpha^b,\;$ there exist zero sequences $\;(a_k)_{k \in \N},\;$ $(\tilde a_k)_{k \in \N} \subset (0, a^b)\;$ such that $\;\omega(a, b)\;$ is differentiable at $\;a_k\;$ and $\;\tilde a_k\;$ and $\;\tilde a_k< a_k\;$ with 

$$\alpha^{(a_k, b)} \longrightarrow \alpha^b \quad \text{ and } \quad \alpha^{(\tilde a_k, b)} \longrightarrow \alpha_{inf} < 0.$$

Because $\;\omega\;$ is continuous,  it attains its minimum on $\;[\tilde a_k, a_k].\;$ Since the minimal Willmore energy $\;\omega(a,b)\;$ is strictly decreasing (with the same arguments as in the proof of $\;\alpha_{sup}\geq 0$) around $\;\tilde a_k\;$ and strictly increasing around $\;a_k\;$ this minimum
 is always attained at $\;\hat a_k \in (\tilde a_k, a_k).\;$ For $\;k\in \N\;$ and $\;a \sim \hat a_k\;$ consider the smooth family of immersions

$$\bar f_k(s(a), t(a)) = \exp _{ f^{(\hat a_k, b)}}\left( s(a)V_1^{(\hat a_k, b)} + t(a) V_2^{(\hat a_k,b)}\right)$$

with 

$$\;\delta \bigPi^i\big(V_i^{(\hat a_k, b)}\big) = \delta_{i,j}\quad \text{ and }\quad\bigPi\Big(\bar f\big(s(a), t(a)\big)\Big) = (a,b)\;$$ as in Step (2). Let 

$$\;\bar \omega_k(a,b)= \mathcal W\Big(\bar f_k\big(s(a), t(a)\big)\Big)\;$$

 be again the Willmore energy of the family $\;\bar f.\;$ Then $\;\del_a \bar \omega_k(a,b) = \alpha^{(\hat a_k, b)} \neq 0.\;$ Thus $\;\bar \omega_k(a,b)\;$ is either strictly increasing or strictly decreasing around $\;\hat a_k\;$ and there  exist an $\;a \sim a_k\;$ and $\;a \in [\tilde a_k, a_k]\;$ with 

\begin{equation}\label{baromega}
\bar \omega_k(a,b) < \bar \omega_k(\hat a_k, b).
\end{equation}

Equation \eqref{baromega} together with the definition of $\;\omega\;$ and $\;\bar  \omega_k\;$ gives a contradiction to the fact that $\;\omega(\hat a_k, b)\;$ is the minimum of $\;\omega\;$ on $\;[\tilde a_k, a_k],\;$ since

$$\omega(a,b) \leq \bar \omega_k(a,b) < \bar \omega_k(\hat a_k,b) = \omega(\hat a_k,b).$$

It remains to show that $\;\alpha_{inf} = \alpha^b.\;$ For this take again a zero sequence $\;(a_k)_{k \in \N} \subset (0, a^b)\;$ with $\;\omega(a,b)\;$ is differentiable at all $\;a_k\;$ and such that corresponding sequence of Lagrange multipliers satisfies $\;\alpha^{(a_k, b)} \longrightarrow \alpha_{inf}.\;$ Thus, as before, we have that up to take a sub sequence and up to invariance $\;f^{(a_k, b)} \longrightarrow f^{(0,b)}\;$ smoothly. If $\;\alpha_{inf} < \alpha^b, \;$ we obtain by Lemma \ref{alpha<alpha^b} that 

$$f^{(a_k, b)} = f^b \quad \text{ for }\quad k>>1$$

 contradicting the fact that $\;a_k > 0.\;$ Thus we can conclude that
 
 $$\alpha_{inf} = \alpha^b.$$
 
\end{proof}


\end{document}